%% file: subgauss_arxiv.tex
\documentclass[final,onefignum,onetabnum]{siamonline220329}

\usepackage[utf8]{inputenc}
\usepackage[T1]{fontenc}
\usepackage{amsmath,amssymb}
\usepackage{algorithm,algorithmic}
\usepackage{hyperref}
\usepackage{graphicx}
\usepackage{caption,subcaption}
\usepackage{stmaryrd}

\input{subgauss_shared}

\begin{document}
\maketitle
\begin{abstract}
This paper expands the analysis of randomized low-rank 
approximation beyond the Gaussian distribution to four classes of random matrices:  (1) independent sub-Gaussian entries, (2) independent sub-Gaussian columns, (3) independent bounded columns, and (4) independent columns with bounded second moment.  Using a novel interpretation of the low-rank approximation error involving sample covariance matrices, we provide insight into the requirements of a \textit{good random matrix} for the purpose of randomized low-rank approximation. 
Although our bounds involve unspecified absolute constants
(a consequence of underlying non-asymptotic theory of random matrices), they allow for qualitative comparisons across distributions.
The analysis offers some details on the minimal number of samples (the number of columns $\ell$ of the random matrix $\mat\Omega$) and the error in the resulting low-rank approximation. We illustrate our analysis in the context of
the randomized subspace iteration method as a representative algorithm for low-rank approximation, however, all the results are broadly applicable to other low-rank approximation techniques. {We conclude our discussion with numerical examples using both synthetic and real-world test matrices.}
\end{abstract}

\section{Introduction}
\label{sec:intro}

Randomized methods have been providing a powerful framework for scientific computations. They are easy to implement, computationally efficient, numerically robust, highly scalable, and very reliable in a vast number of practical applications, making them a game-changing tool in scientific computing and data analysis. Randomized algorithms have already left their footprint in numerical linear algebra (NLA) (see, e.g., the survey articles~\cite{halko2011finding,mahoney2011randomized,woodruff2014sketching,martinsson2020randomized,murray2023randomized}). In this paper, we focus our attention {on} randomized algorithms for low-rank approximations. Given an $m \times n$ matrix $\mat A$ the task is to find an $m \times k$ matrix $\mat B$ and a $k \times n$ matrix $\mat  C$ such that
\[
\mat A \approx \mat B \mat C.
\]

There are several algorithms to accomplish this such as randomized subspace iteration, Nystr\"om, and block Krylov-based methods. However, the current state-of-the-art in these approaches is to focus the analysis primarily on standard Gaussian random matrices or other special cases. There are several reasons for this. The analysis for Gaussian random matrices uses well-known results from random matrix theory where the estimates are sharp and all the constants are fully specified. This makes it possible to derive quantitative error bounds as well as practical recommendations for the selection of {parameters}  (e.g., oversampling parameters). From a computational perspective, generating Gaussian random matrices has a fixed cost for each entry of the random matrix. However, for very large matrices generating and storing Gaussian random matrices can be prohibitively expensive.

This computational issue is particularly relevant for randomized low-rank tensor decompositions. To illustrate this, consider an order $d$ tensor of size $n\times \dots \times n$. A prototypical algorithm --- the randomized higher-order singular value decomposition --- applies the randomized SVD to $d$ different mode-unfoldings each of size $n \times n^{d-1}$. {It requires generating}  random matrices of size $n^{d-1} \times r$, where $r$ is the target rank, which can be infeasible even for modest values of $n$ and $d$. In such situations, current research is exploring random matrices that are structured such as subsampled trigonometric transforms, {e.g., subsampled randomized Hadamard transform (SRHT)}~\cite{tropp2011improved}, random tensor products~\cite{che2019randomized,minster2022parallel}, and random tensor-train~\cite{al2023randomized}, to name just a few. In these examples, the random matrices are not constructed explicitly, so the cost of forming and storing the matrices is much cheaper, and additionally in some circumstances {the application of the matrix} can be performed efficiently.

Therefore, there is a clear need for exploring non-Gaussian random matrices. An ideal random matrix is one that should be easy to construct, store, obtain the sketch, and yields comparable errors (in theory and practice) to the Gaussian case. There are several gaps in the analysis of low-rank approximations for non-Gaussian random matrices. First, in previous work, the analysis typically handles specific random matrices and not classes of random matrices. Second, for many random matrices, there is no analysis at all, but numerical evidence suggests behavior similar to Gaussian random matrices. Third, the lack of a unified perspective makes it difficult to compare different random matrices on the same footing. Finally, the lack of understanding of the necessary and sufficient conditions of a random matrix makes it difficult to construct new random matrices with favorable properties. With recent advances in non-asymptotic random matrix theory~\cite{edelman2005random,vershynin2010introduction,tao2012topics,vershynin2018high,tropp2015introduction}, we develop analysis for randomized algorithms that use non-Gaussian random matrices.

To address these gaps, we provide analysis and numerical experiments for several families of random matrices {$\mat{\Omega} \in \R^{n \times \ell}$}: (i) independent sub-Gaussian entries, (ii) independent sub-Gaussian columns, (iii) independent bounded columns, {(iv) independent columns with bounded second moment}.  Our analysis provides a way forward for systematic comparison and tools for analyzing new random matrices that fit within one of the four classes of random matrices. The focus of this paper is on the analysis and \textit{not} on the computational issues such as the cost of constructing, storing, and computing with the random matrices. We summarize the main contributions and features of our work. 

\paragraph{Contributions and Outline}
In this paper, we expand the analysis of randomized low-rank approximations to account for many distributions. Our goal is to answer the question: What are the characteristics of a good random matrix distribution for use in randomized low-rank approximations? The specific contributions and features of our work are as follows:
\begin{enumerate}
\item We give a novel interpretation of the error in the low-rank approximation (Section~\ref{ssec:interp}) using sample covariance approximations that provide insight into the requirements of a random matrix.
    \item We give sufficient conditions for four classes of random matrices: (1) independent sub-Gaussian entries (\cref{thm:subgauss}), (2) independent sub-Gaussian columns (\cref{thm:subgauss2}), (3) independent bounded columns (\cref{thm:genindep}), and (4) independent columns with bounded second moment (\cref{thm:indepinexp}). 
    \item We illustrate our analysis for the randomized subspace iteration which is a representative algorithm for low-rank approximation, but our analysis is more broadly applicable to other algorithms such as Nystr\"om, block-Krylov, eigenvalue approximations, etc. 
    \item {The two metrics that will be used for comparison purposes} are the minimal number of samples required (i.e., the number $\ell$ of columns of the random matrix $\mathbf{\Omega}$) and the error in the resulting low-rank approximations.

    \item Our analysis has an emphasis on developing bounds that do not explicitly depend on the size of the matrix (by the use of the concept of stable rank) which makes it applicable to large-scale problems. Our bounds involve unspecified absolute constants (since they are based on results for which the constants are not optimal or unavailable). These bounds still allow for a qualitative comparison between different random matrices. 
    \item The novel results provide analytical justification for many distributions of random matrices. Specifically, they unify several existing results (for individual distributions) and greatly expand the types of distributions that can be used in low-rank computations.   
    \item We present numerical experiments (Section~\ref{sec:num}) on several test matrices and for several distributions of random matrices, many of which are unexplored in the context of low-rank approximations.
\end{enumerate}
{We point out that
the main results are clearly presented and discussed in Sections~\ref{sec:intro}--\ref{sec:main}, while for enhanced readability all proofs are deferred till Section~\ref{sec:proofs}}.

\paragraph{Previous work and Applications of presented results} 
Our analysis is not only relevant for the randomized subspace iteration but can be applied more broadly to a host of algorithms for low-rank approximation such as Nystr\"om, subspace iteration, block-Krylov, single view, and streaming algorithms. Beyond low-rank approximation, presented results can be used in the error analysis of Hermitian and generalized Hermitian eigenvalue problems. We briefly list the random matrices which are covered by our analysis and those which are not. Our first result,~\cref{thm:subgauss},  applies to many random matrices including Gaussian, Rademacher, and uniform distribution. Of these, analysis for Gaussian random matrices is most prevalent~\cite{halko2011finding,martinsson2020randomized}, followed by Rademacher~\cite{saibaba2017randomized}.  Our second result,~\cref{thm:subgauss2},  is relevant to distributions such as the column spherical distribution. We believe (to the best of our knowledge) that the analysis for these distributions is currently unavailable. Our third result,~\cref{thm:genindep}, is applicable to coordinate distributions, including leverage score sampling, and random frames. The analysis for these distributions is prevalent (see, e.g.,~\cite[Section 9.6.3]{martinsson2020randomized}) but is limited to specific distributions, whereas our analysis is more general. To our knowledge, the final main result \cref{thm:indepinexp}, has not been considered in the literature and greatly expands the types of distributions for which the analysis can be used such as log-concave, $\alpha-$sub-exponential, which in turn cover many distributions. A related model was analyzed in~\cite{oymak2018universality}.

\paragraph{Random matrices not covered} Our analysis, while extensive, is not suitable for many random matrices, as we now explain. All presented results require at the very least finite second moments, which excludes distributions such as the Cauchy distribution. 
Furthermore, our analysis does not cover random matrices with independent rows, even if they are sub-Gaussian. The main difficulty {in this case comes from the fact that if} the random matrix $\mat\Omega$ has independent rows, we need to analyze the singular values of matrix $\mat{W}^\top\mat\Omega$ (where $\mat{W}$ has orthonormal columns), which does not have independent rows. Examples of these distributions include (1) the sparse random sign matrix, whose rows have independent random signs in a selected number of rows~\cite{martinsson2020randomized,cohen2016nearly,tropp2019streaming}; (2) random matrix with spherically distributed rows, which has been considered in the diagonal estimation literature but thus far has not been considered for low-rank approximation purposes. A potential way to tackle this class of random matrices is to use the decoupling technique in~\cite[Section 5.5]{vershynin2010introduction} and~\cite[Chapter 6.1]{vershynin2018high}{, but that is not the focus of this paper}.  Other special cases of great importance not covered by our analysis are the subsampled trigonometric transforms (SRTTs) (see e.g.,~\cite{tropp2011improved,halko2011finding,martinsson2020randomized} and references within) whose analysis bears some resemblance to the coordinate/random frame distributions and the {\sf CountSketch} matrix~\cite{charikar2004finding,clarkson2017low} which have neither independent rows nor columns. 

\section{Background}
\label{sec:background}
In this section, we collect the necessary background information on sub-Gaussian random variables and vectors (Section~\ref{ssec:subgauss}),  and randomized algorithms for low-rank approximation (Sections~\ref{ssec:randsvd} and~\ref{ssec:setup}). {A word on constants:} In the discussion below, we denote by $c, C, C', C''$, etc. absolute constants. These constants do not depend on the spectrum of the matrix or the distribution of the random variables. We cannot specify these constants explicitly since our analysis depends on {non-asymptotic matrix theory} results for which these constants are either not specified or sometimes unknown. However, in the following, we do our best to track the dependence on all constants {involved}.

\subsection{Sub-Gaussian distributions}\label{ssec:subgauss}

In this section, we review the definitions of a sub-Gaussian distribution and the various other distributions that fit into the family of sub-Gaussian distributions. This discussion is closely tied to~\cite{vershynin2010introduction,vershynin2018high}. The sub-Gaussian distribution is a rich class of probability distributions whose tails decay as sharply as Gaussian random variables. Besides Gaussian random variables, it contains distributions such as:
\begin{enumerate}
    \item \textbf{Rademacher}: The random variable takes the value $\pm 1$ with probability $1/2$ each. This is a centralized (symmetric) version of the Bernoulli random variable.
    \item \textbf{Sparse Rademacher}: This is a distribution parameterized by a sparsity parameter $s \geq 1$. The random variables take the values $\displaystyle \{-\sqrt{s},0,\sqrt{s}\}$ with probabilities respectively $\displaystyle \{1/(2s), 1-1/s, 1/(2s)\}$. Note that $s = 1$ gives the Rademacher distribution.
    \item \textbf{Uniform}: The random variable is uniform in the interval $\displaystyle [-\sqrt{3},\sqrt{3}]$, with mean zero and unit variance (denoted $\mc{U}([-\sqrt{3},\sqrt{3}])$).
    \item \textbf{Bounded}: The random variable $X$ is such that $|X| \leq M$ almost surely, where $M$ is finite. It generalizes the earlier listed random variables.  
\end{enumerate}

The following collects various definitions involving sub-Gaussian random variables; see~\cite[Chapters 2 and 3]{vershynin2018high}.
We say that a random variable $X$ is {\em sub-Gaussian} if $\prob{|X| \geq t} \leq 2\exp(-t^2/\sigma^2)$ for all $t\geq 0$ and finite $\sigma$. There are other ways of defining the sub-Gaussian random variables using the equivalence in~\cite[Proposition 2.5.2]{vershynin2018high}.  The sub-Gaussian norm of the {random variable} is defined as  \begin{equation}\label{eqn:subgaussnorm}\|X\|_{\psi_2} := \mathsf{inf}\{t > 0 \ | \ \expect{\exp(X^2/t^2)} \leq 2 \}.\end{equation}

For the distributions listed above, the sub-Gaussian norms are respectively $\mc{O}(1)$ (Rademacher), $\mc{O}(\sqrt{s})$ (sparse Rademacher), and $\mc{O}(M)$ (bounded random variables $|X| \leq M$ almost surely). Although it is natural to consider random matrices with independent sub-Gaussian entries (see \cref{def:randommat1}), in some cases, the entries of the random matrix may not be independent. 

To this end, we first define random sub-Gaussian vectors. A random vector $\vec{x} \in \R^n$ is said to be sub-Gaussian if  $\vec{x}^\top \vec{z}$ is {\em sub-Gaussian} for all $\vec{z} \in \R^n$. Hence, the sub-Gaussian norm of $\vec{x}$ is defined as $\|\vec{x}\|_{\psi_2} : = \sup_{\vec{z} \in \mc{S}^{n-1}} \| \vec{x}^\top \vec{z}\|_{\psi_2}$, where $\mc{S}^{n-1} := \{ \mathbf{x} \in \R^n : \|\mathbf{x}\|_2 = 1\}$ is the unit Euclidean sphere in $\R^n$.
A random vector $\vec{x} \in \R^n$ is said to be {\em isotropic} if $\expect{\vec{x}} = \vec{0}$ and $\expect{\vec{xx}^\top} = \mat{I}$; alternatively, it is isotropic if $\expect{\inner{\vec{x}}{\vec{z}}^2} = \inner{\vec{z}}{\vec{z}}^2$ for every $\vec{z} \in \R^n$.

 The following random vectors are examples that we will consider in this work (see~\cite[Chapter 3]{vershynin2018high} for more discussion):
\begin{enumerate}
    \item \textbf{Spherical}: A {random vector $\vec{x}$} which is uniformly distributed on the Euclidean sphere of radius $\sqrt{n}$ and centered at the origin, i.e., $\vec{x} \sim \mbox{Unif}(\sqrt{n} \mc{S}^{n-1})$. To generate a random vector from this distribution, take a random Gaussian vector $\vec{g} \sim \mc{N}(\vec{0},\mat{I})$ and compute $\vec{x}  =\sqrt{n} \vec{g}/\|\vec{g}\|_2$. Then $\vec{x} \sim \text{Unif}(\sqrt{n} \mc{S}^{n-1})$ and $\|\vec{x}\|_{\psi_2} \leq C$, where $C$ is an absolute constant; see~\cite[Theorem 3.4.6]{vershynin2018high}. 
    \item \textbf{Random orthogonal}: The vectors are selected uniformly from the set $\{ \sqrt{n}\vec{q}_i \}_{i=1}^n$, where $\vec{q}_i$ are the columns of an orthogonal matrix. This is a special case of the random frame distribution, in which the vectors form a tight frame with equal norms. Another special case is the coordinate distribution where $\{\vec{q}_i\}_{i=1}^n$ are the columns of the identity matrix. 
    
\end{enumerate}
{We will also consider random matrices with independent sub-Gaussian columns (see~\cref{def:randommat2}), independent columns with bounded norms (see~\cref{def:randommat3})
and independent columns with bounded second moments (see~\cref{def:randommat4})}.

\subsection{Randomized SVD}\label{ssec:randsvd}

Let us briefly describe how to determine a low-rank approximation of the matrix $\mathbf{A} \in \R^{m \times n}, m \geq n$, using a randomized method, i.e., randomized subspace iteration algorithm.

Given a random matrix $\mat{\Omega} \in \R^{n \times \ell}$, we carry out $q$ steps of the randomized subspace iteration method to obtain the matrix $\mat{Y}$, also known as the \textit{sketch} of the input matrix $\mat{A}$. Then, we perform a thin-$QR$ factorization of sketch $\mat{Y} \in \R^{m \times \ell}$ to obtain $\mat{Q} \in \R^{m \times \ell}$
whose columns form an orthonormal basis for the range of $\mat{Y}$. The main idea is that, under suitable conditions~\cite{halko2011finding}, the range of $\mat{Q}$ is a good approximation for the range of $\mat{A}$. Finally, we obtain a low-rank approximation of  $\mat{A}$ by the projection $\mat{QQ}^T\mat{A}$. The rest of the algorithm involves converting this low-rank approximation into the SVD format.
 The pseudocode of this basic algorithm is given in \Cref{alg:randsvd}. We call~\Cref{alg:randsvd}
\textit{idealized}, since it can behave poorly in the presence of round-off errors. For details on a more practical implementation of this algorithm, we refer the reader to~\cite{halko2011finding,saad2011numerical}.
\begin{algorithm}
\begin{algorithmic}[1] 
    \REQUIRE Matrix $\mat{A} \in \R^{m\times n}$ with $m \geq n$,
    random matrix $\mathbf{\Omega} \in \R^{n \times \ell}$, $\#$ of subspace iteration steps $q \geq 0$ (integer).
    \STATE Compute $\mat{Y} = (\mat{AA}^\top)^q \mat{A\Omega}$
    \STATE Compute $\mat{Y} = \mat{QR}$ \qquad \quad \texttt{/* thin QR factorization */}
    \STATE Compute $\mat{B} = \mat{Q}^\top \mat{A}$  \qquad \ \texttt{/* small matrix */}\\[0.02in]
    \STATE Compute $\mat{B} = \mat{U}_B\mathat{\Sigma}\mathat{V}^\top$ \ \ \texttt{/* thin SVD factorization */}\\[0.02in]
    \STATE Set $\mathat{U} = \mat{QU}_B$
    \ENSURE $\mathat{U}, \mathat{\Sigma}, \mathat{V}$ such that $\mathbf{A} \approx \mat{Q}\mat{Q}^T\mat{A} = \mathat{U}\mathat{\Sigma}\mathat{V}^T$.
\end{algorithmic}
\caption{Randomized SVD (Idealized version)} 
\label{alg:randsvd}
\end{algorithm}

\subsection{Setup for the analysis}\label{ssec:setup}
Let $\mat{A} \in \R^{m\times n}$ with $m \geq n$ has a singular value decomposition (SVD) of the form 
\[ \mat{A} = \mat{U\Sigma V}^\top = \bmat{\mat{U}_k & \mat{U}_\perp } \bmat{ \mat\Sigma_k \\ & \mat\Sigma_\perp } \bmat{\mat{V}_k^\top \\ \mat{V}_\perp^\top},\] 
where $1\leq k \leq \rank(\mat{A})$.
Here, $\mat{\Sigma}_k \in \R^{k \times k}$ and $\mat{\Sigma}_\perp \in \R^{(m-k) \times (n-k)}$, the columns of $\mat{U}_k$ and $\mat{U}_\perp$ are the corresponding left singular vectors, and columns of $\mat{V}_k$ and $\mat{V}_\perp$ are the corresponding right singular vectors.
The singular values $\sigma_i(\mathbf{A}), 1\leq i \leq n$, of $\mat{A}$ are arranged in decreasing order such that the diagonal elements of matrix $\mat\Sigma_k$ are the dominant singular values of the matrix $\mat{A}$. We assume that there is a gap in the singular values at index $k$, i.e., 
\[
\sigma_1(\mat{A}) \; \geq \;\sigma_2(\mat{A}) \; \geq \cdots \; \geq
\sigma_k(\mat{A}) \; > \; \sigma_{k+1} (\mat{A}) \; \geq \cdots \; \geq \; \sigma_n(\mat{A}) \geq 0,
\]
and we define the singular value ratios 
\begin{equation}
    \gamma_j := \frac{\sigma_{k+1}(\mat{A})}{\sigma_j(\mat{A})},  \qquad 1 \leq j \leq k.
\end{equation}
With the singular values monotonically decreasing, the singular value ratios satisfy $\gamma_1 \leq \dots \leq \gamma_k < 1$. Next we define the \textit{stable rank} of a nonzero matrix $\mat{A}$ as  $\sr(\mat{A}) := {\| \mat{A} \|_F^2}/{\|\mat{A}\|_2^2 }$.
The stable rank is a proxy for the rank of the matrix and it satisfies the following inequality $1 \leq \sr(\mat{A}) \leq \rank(\mat{A})$.

We partition the matrix $\mat{V}\t \mat\Omega \in \R^{n \times \ell}$ into two parts, i.e.,
\begin{equation}
    \Oh_1 := \mat{V}_k^\top \mat\Omega \in \R^{k \times \ell} \qquad \mbox{ and } \qquad \Oh_2 :=  \mat{V}_\perp^\top \mat\Omega \in \R^{(n-k) \times \ell},
\end{equation}
where numbers $k$ and $\ell$ satisfy 
$ k \leq \ell \leq n$.
{We consider a bound for the structural analysis of the subspace iteration method.} The following result quantifies the error in the low-rank approximation.
\begin{theorem}[Structural bound]
\label{thm:StrBound}
Let $\mat{A} \in \R^{m\times n}$ with $m \geq n$ and $\mathbf{Q} \in \R^{m \times \ell}$ be the output of~\Cref{alg:randsvd} with the number of subspace iteration steps $q$. Then
    \[ \|(\mat{I} - \mat{QQ}^\top) \mat{A}\|_2^2 = \| \mat{A} - \mathat{U}\mathat{\Sigma}\mathat{V}\t\|_2^2  \leq \|\mat\Sigma_\perp\|_2^2 + \gamma_k^{4q}\|\mat\Sigma_\perp\Oh_2\Oh_1^\dagger\|_2^2. \]
\end{theorem}
\begin{proof}
    See~\cite[Theorem 8]{saibaba2017randomized}.
\end{proof}
The interpretation of this theorem is that the error in the low-rank approximation is governed by the tail of the singular values in matrix $\mat\Sigma_\perp$. The second term is small if the gap $1/\gamma_k$ is large or the number of subspace iterations steps $q$ is large. The influence of the random matrix (also starting guess) is captured in the term $\|\mat\Sigma_\perp\Oh_2\Oh_1^\dagger\|_2$. In Section~\ref{ssec:interp}, we provide a novel interpretation of this term that is useful in the upcoming analysis.

Terms related to $\|\mat\Sigma_\perp\Oh_2\Oh_1^\dagger\|_2$ repeatedly appear in the analysis of randomized low-rank approximations: randomized SVD low-rank approximation (\cite[Theorem 9.1]{halko2011finding}), accuracy of singular values and singular vectors~\cite[Theorems 1 and 2]{saibaba2017randomized}, Nystr\"om~\cite[Proof of Theorem 2.1]{tropp2019streaming}, block Krylov subspace methods (\cite[Theorem 2.1]{drineas2018structural}). Therefore, by analyzing this term for various distributions, we can derive results for low-rank approximation for other algorithms. However, in this paper, we focus on the randomized subspace iteration as a representative algorithm.

The following provides a bound for the term $\|\mat\Sigma_\perp\Oh_2\Oh_1^\dagger\|_2$ when the random matrix $\mat\Omega \in \R^{n \times \ell}$ is a standard Gaussian random matrix. We assume throughout $\mat\Sigma_\perp \neq \mat{0}$.  

\begin{lemma}\label{lem:gauss}
Let $\mat{\Omega} \in \R^{n \times \ell}$ be a standard Gaussian matrix, such that $\ell := k + p \leq n$ with an oversampling parameter $p\geq 4$.  Then for all \ $0 < \delta < 1$, with probability of failure at most $\delta$, we have 
\[
\|\mat{\Sigma}_\perp\Oh_2\Oh_1^\dagger\|_2 \leq   \|\mat\Sigma_\perp\|_2 \left( \frac{4}{\delta}\right)^{1/p}\left[\sqrt{\frac{3k}{p+1 }}   + \frac{e\sqrt{\ell}}{p+1} \left( {\Delta} +   \sqrt{\sr(\mat\Sigma_\perp)} \right)\right],
\]
where  ${\Delta} := \sqrt{2\log\left(\frac{2}{\delta}\right)}$.
\end{lemma}
\begin{proof}
The proof follows from the one of~\cite[Theorem 10.8]{halko2011finding}. We provide the details in Section~\ref{ssec:aux}.
\end{proof}

\section{Main results}\label{sec:main}
We first present a novel interpretation of the term $\|\mat\Sigma_\perp\Oh_2\Oh_1^\dagger\|_2$ using sample covariance matrices. Next, we derive bounds for this term and the low-rank approximations for each of the four classes of random matrices that are considered in this paper. The bounds are illustrated for a few example distributions, but other examples are provided in the Supplementary Materials {(\cref{ssec:otherex})}.
\subsection{Connection to covariance estimation}\label{ssec:interp}  { We give an intuitive explanation of the sufficient conditions required of a random matrix by drawing connections to the covariance estimation problem.} Consider random vectors $\vec{x}_j$ with zero mean and covariance $\mat\Gamma$. Given independent draws $\vec{x}_1,\dots, \vec{x}_N$, define the sample covariance matrix 
\[ \mathat{\Gamma}_N := \frac{1}{N} \sum_{j=1}^N \vec{x}_j\vec{x}_j\t.\]
It is easy to verify that $\mathat{\Gamma}_N$ is an unbiased estimator for the covariance matrix $\mat\Gamma$. A natural question is to find the number of samples $N$ such that the relative error for the sample covariance matrix $\mathat{\Gamma}_N$ satisfies 
\[ \|\mathat{\Gamma}_N - \mat\Gamma\|_2 \leq \eta \|\mat\Gamma\|_2, \]
 with high probability. This question has a rich history in statistics, see~\cite[Section 5.4.3]{vershynin2010introduction} and~\cite[Chapter 5.6]{vershynin2018high}.

To see the connection with the covariance estimation problem, we consider the following lemma and, for the moment, deterministic $\mat\Omega$.  
\begin{lemma}\label{lem:samplecov}
    Suppose that $\mat{Z} = \mat{\Sigma}_\perp\Oh_2$ and
    \begin{equation}\label{eqn:covestimates}
        \|\frac{1}{\ell}\mat{ZZ}\t - \mat\Sigma_\perp\mat\Sigma_\perp\t\|_2 \leq \eta \|\mat\Sigma_\perp\|_2^2,\qquad \|\frac1\ell\Oh_1\Oh_1\t - \mat{I}\|_2 \leq \varepsilon,
    \end{equation}  
    for $\eta > 0$ and $\varepsilon\in(0,1)$. Then 
    \[ \|\mat{\Sigma}_\perp\Oh_2\Oh_1^\dagger\|_2^2 \leq \frac{ \|\frac{1}{\ell}\mat{ZZ}\t - \mat\Sigma_\perp\mat\Sigma_\perp\t\|_2 + \|\mat\Sigma_\perp\|_2^2 }{1 - \|\frac1\ell\Oh_1\Oh_1\t - \mat{I}\|_2 }   \leq \|\mat\Sigma_\perp\|_2^2 \frac{1+\eta}{1-\varepsilon}. \]
\end{lemma}
\begin{proof}
See Section~\ref{ssec:aux}.
\end{proof}
Now assume that the columns of $\mat\Omega$ are independent, zero mean, and isotropic. Therefore, by the linearity of expectations, $$\expect{\mat{ZZ}\t} = \ell \mat\Sigma_\perp\mat\Sigma_\perp\t \qquad  \expect{\Oh_1\Oh_1\t} = \ell \mat{I}.$$ The terms $\frac{1}{\ell}\mat{ZZ}\t$ and $\frac{1}{\ell}\Oh_1\Oh_1\t$ are sample covariance matrices for $\mat\Sigma_\perp\mat\Sigma_\perp\t$ and the $k\times k$ identity matrix, respectively. Then, $\|\frac{1}{\ell}\mat{ZZ}\t - \mat\Sigma_\perp\mat\Sigma_\perp\t\|_2$ \ and \ $\|\frac1\ell\Oh_1\Oh_1\t - \mat{I}\|_2$  represent the relative error in the sample covariance approximations. We want to find the minimal number of samples $\ell$ such that the two bounds in~\eqref{eqn:covestimates} hold simultaneously with high probability. As we show, for many distributions $\ell = \mc{O}(k)$ samples are sufficient, whereas for other distributions $\mc{O}(k\log k)$ samples are necessary. Furthermore, this connection to covariance estimation gives insight into why all our assumptions require the columns to be independent, zero mean,  and isotropic.

\subsection{Independent sub-Gaussian entries}\label{ssec:subgauss1}

In this section, we derive a result analogous to \cref{lem:gauss} for standard sub-Gaussian random matrices with independent entries. We define a model of the random matrix to be analyzed and use it to introduce the main result.

\begin{definition}[Independent sub-Gaussian entries] \label{def:randommat1}
In this model, we consider the random matrix $\mat\Omega \in \R^{n\times \ell}$ with the following properties:
   \begin{enumerate}
    \item Independence: the entries of $\mat{\Omega}$ are independent copies of the sub-Gaussian random variable $X$.
    \item Normalization: {the random variable} $X$ has zero mean and unit variance.
    \item Boundedness: {the random variable} $X$ satisfies {$\|X\|_{\psi_2} \leq K_{\rm E}$ with $K_E > 0$}.
\end{enumerate} 
\end{definition}

\begin{theorem}[Independent sub-Gaussian entries]\label{thm:subgauss}
Let $\mat{\Omega} \in \R^{n \times \ell}$ be a random matrix satisfying \cref{def:randommat1}. Let $\varepsilon,\delta \in (0,1)$ be user-specified parameters, and let  the number of samples satisfy
\begin{equation}
    \label{eqn:ellbound}n \geq \ell =  \frac{C_{\rm ES}K_{\rm E}^4}{\varepsilon^2} (\sqrt{k} + {V_\delta})^2  =  \mc{O}(K_{E}^4k),
\end{equation}
 where $V_\delta := \sqrt{\log(4/\delta)}$.
Then, with a probability of failure at most $\delta$, we have
\begin{equation*} \|\mat\Sigma_\perp\Oh_2\Oh_1^\dagger\|_2  \leq   \frac{C_{\rm EB}K_{\rm E}\|\mat\Sigma_\perp\|_2}{(1-\varepsilon)\sqrt{\ell}}\left( \sqrt{\sr(\mat\Sigma_\perp)} + \sqrt{k} +  {V_\delta}  \right). \end{equation*}
If  $\mat\Omega$ is an input to~\cref{alg:randsvd}, then with a probability of failure at most $\delta$, 
\[\|\mat{A} - \mathat{U}\mathat{\Sigma}\mathat{V}\t \|^2 \leq  \|\mat\Sigma_\perp\|_2^2 + \frac{C_{\rm EB}^2 K_{\rm E}^2 \gamma_k^{4q} \|\mat\Sigma_\perp\|_2^2}{(1-\varepsilon)^2{\ell}}\left( \sqrt{\sr(\mat\Sigma_\perp)} + \sqrt{k} +  {V_\delta}  \right)^2 . \]
Here, $C_{\rm ES}$ and $C_{\rm EB}$ are absolute constants.
\end{theorem}
\begin{proof}
See Section~\ref{subsec:sub-Gaussian}.
\end{proof}
The result of~\Cref{thm:subgauss} is remarkable in the sense that it is applicable to any sub-Gaussian random matrix $\mat\Omega \in \R^{n \times \ell}$ as long as it comes from the model presented in~\Cref{def:randommat1}. It also predicts that the number of required samples $\ell$ is $\mc{O}(k)$. If $\B\Omega \in \R^{n \times \ell}$ is a standard Gaussian random matrix, then $K_{\rm E}$ is an absolute constant. This result is comparable with~\cite[Theorem 10.8]{halko2011finding} up to unspecified constants; \cref{thm:subgauss} has a better dependence on the failure probability.

 We discuss the implications for the other distributions mentioned in Section~\ref{ssec:subgauss}. Consider the case when $\B\Omega \in \R^{n \times \ell}$ has i.i.d entries from the sparse Rademacher distribution with sparsity parameter $s$. The entries have a sub-Gaussian norm {bounded by} $K_{\rm E} = \sqrt{s}$, so the required number of samples is $\mc{O}(s^2k)$. Thus, the larger the sparsity parameter $s$, the more samples are required. In particular, if $s=1$ (Rademacher distribution), then the required number of samples is $\mc{O}(k)$. Our result is an improvement on~\cite{saibaba2017randomized}, which suggests $\mc{O}(k\log n)$ samples are required and which appears to overestimate the {required number} of samples by a factor of $\log n$. Similarly, the uniform distribution $\mc{U}[-\sqrt{3},\sqrt{3}]$ also requires $\mc{O}(k)$ samples.

\subsection{Independent sub-Gaussian columns} The next result is applicable to sub-Gaussian matrices with a weaker assumption, namely, the individual entries {do not need to be independent}, only the columns need to be independent. We summarize this {random matrix model} in the following definition.
\begin{definition}[Independent sub-Gaussian columns] \label{def:randommat2}
In this model, we consider the random matrix $\mat\Omega \in \R^{n\times \ell}$ with the following properties:
   \begin{enumerate}
    \item Independence: the columns of matrix $\mat\Omega \in \R^{n \times \ell}$ are independent copies of the sub-Gaussian random vector $\vec{x} \in \R^n$.
    \item Normalization: the random vector $\vec{x} \in \R^n$ is isotropic.
    \item Boundedness: the random vector $\vec{x} \in \R^n$ satisfies $\|\vec{x}\|_{\psi_2} \leq K_C$, with $K_C > 0$.
\end{enumerate} 
\end{definition}

\begin{theorem}[Independent sub-Gaussian columns]\label{thm:subgauss2} Let $\mat{\Omega} \in \R^{n \times \ell}$ be a random matrix satisfying \cref{def:randommat2}. Let $\varepsilon, \delta \in (0,1)$ be user-defined parameters, and  let the number of samples satisfy 
\begin{equation}
    \label{eqn:ellbound2} n \geq \ell =\frac{C_{\rm CS}K_C^4}{\varepsilon^2} (\sqrt{k} + {V_\delta})^2  =  \mc{O}(K_C^4k).
\end{equation}
Then, with a probability of failure at most $\delta$, we have
\[ \|\mat\Sigma_\perp\Oh_2\Oh_1^\dagger\|_2  \leq   \frac{ \|\mat\Sigma_\perp\|_2}{(1-\varepsilon)\sqrt{\ell}}\left( \sqrt{\ell} + C_{\rm CB}K_{\rm C}^2(\sqrt{\sr(\mat\Sigma_\perp) } +  {V_\delta})  \right). \]
If  $\mat\Omega$ is an input to~\cref{alg:randsvd}, then with a probability of failure at most $\delta$, 
\[\|\mat{A} - \mathat{U}\mathat{\Sigma}\mathat{V}\t \|^2 \leq \|\mat\Sigma_\perp\|_2^2  + \frac{ \gamma_k^{4q}\|\mat\Sigma_\perp\|_2^2}{(1-\varepsilon)^2\ell}\left( \sqrt{\ell} + C_{\rm CB}K_{\rm C}^2(\sqrt{\sr(\mat\Sigma_\perp) } +  {V_\delta})  \right)^2. \]
Here, $C_{\rm CS}$ and $C_{\rm CB}$ are absolute constants.
\end{theorem}
\begin{proof}
See Section~\ref{subsec:sub-Gaussian-col}.
\end{proof}

The interpretation of the result of~\cref{thm:subgauss2} is very similar to that of~\Cref{thm:subgauss}. The two results are comparable since a random matrix with independent entries has also independent columns, however, the converse does not need to be true. A direct comparison though is not possible, since the precise constants are not explicitly known but the two results differ in their dependence on the sub-Gaussian norms.

Beyond independent entries, we now consider a few other special cases. If the columns of $\mat\Omega \in \R^{n \times \ell}$ are i.i.d spherical Gaussian random vectors, or if they are drawn from certain isotropic convex sets in $\R^n$ (called $\psi_2$ bodies,~\cite[Section 5.2.5]{vershynin2010introduction}), the {sub-Gaussian norms $K_C$ are absolute constants}.  In both cases, the number of required samples is $\mc{O}(k)$. In Supplementary Materials (Section~\ref{ssec:supp_mat2}) we investigate another distribution called sparse sign matrix.

We briefly discuss the coordinate distribution. Although the coordinate distribution also has independent sub-Gaussian columns, the conclusions of \cref{thm:subgauss2} are very pessimistic since the random vectors have large sub-Gaussian norms, i.e., $K_C = \sqrt{n}$. We address this issue next using a different technique.

\subsection{Independent bounded columns}

We now consider the case that the random matrix $\mat\Omega \in \R^{n \times \ell}$ has independent columns that are not necessarily sub-Gaussian. We, however, require the columns to be bounded almost surely. 

\begin{definition}[Independent bounded columns] \label{def:randommat3}
In this model, we consider the random matrix $\mat\Omega \in \R^{n\times \ell}$ with the following properties:
   \begin{enumerate}
    \item Independence: the columns of the random matrix $\mat\Omega \in \R^{n \times \ell}$ are independent copies of a random vector.
    \item Normalization: each column is isotropic.
    \item Boundedness: the columns satisfy 
    \[  \max_{1 \leq j \leq \ell}\|\mat{V}_k^\top \B\Omega\vec{e}_j\|_2 \leq K_k, \qquad  \max_{1 \leq j \leq \ell}\|\mat{V}_\perp^\top \B\Omega\vec{e}_j\|_2  \leq K_\perp,  \quad K_k > 0, \]
almost surely. \end{enumerate} 
\end{definition}

\begin{theorem}[Independent bounded columns]\label{thm:genindep} Let $\mat{\Omega} \in \R^{n \times \ell}$ be a random matrix  satisfying \cref{def:randommat3} with $K_\perp\geq 1$. Let $\varepsilon, \delta\in (0,1)$ be user-defined parameters, and suppose that the number of samples satisfies
\begin{equation}
    \label{eqn:ellbound3} n \geq \ell =\frac{2 K_k^2}{\varepsilon^2} \log(4k/\delta)  =  \mc{O}(K_k^2 \log(k)).
\end{equation}
Then,  with a probability of failure at most $\delta$, we have
\[ \|\mat\Sigma_\perp\Oh_2\Oh_1^\dagger\|_2  \leq   \frac{ \|\mat\Sigma_\perp\|_2}{\sqrt{(1-\varepsilon) 
\ell}}\left( 3\ell + K_\perp^2/3 \log ({16\sr(\mat\Sigma_\perp)}/{\delta})\right)^{1/2}. \]
If  $\mat\Omega$ is an input to~\cref{alg:randsvd}, then with a probability of failure at most $\delta$, 
\[\|\mat{A} - \mathat{U}\mathat{\Sigma}\mathat{V}\t \|^2 \leq \|\mat\Sigma_\perp\|_2^2  + \frac{\gamma_k^{4q} \|\mat\Sigma_\perp\|_2^2 }{(1-\varepsilon) \ell}\left( 3\ell + K_\perp^2 \log ({16\sr(\mat\Sigma_\perp)}/{\delta})\right). \]
\end{theorem}

\begin{proof}
See Section~\ref{subsec:bounded-col}.
\end{proof}

A nice feature of \cref{thm:genindep} is that, unlike \cref{thm:subgauss,thm:subgauss2}, all the constants involved are known explicitly. However, care should be taken while applying these results since the parameters $K_k$ and  $K_\perp$, although may depend on the dimensions of the problem, are independent of the spectrum of the matrix. To illustrate the result of~\Cref{thm:genindep}, we consider the coordinate distribution (Section~\ref{ssec:supp_mat3}) and leverage score sampling.

\paragraph{Leverage score distribution} To address the issue with the coordinate distribution, a remedy proposed in the literature was to use an importance sampling approach based on the squared row norms of $\mat{V}_k$, known as the (subspace) leverage scores
\begin{equation}
    \label{eqn:lev}
    p_i^{\rm lev} \equiv \frac{\|\mat{V}_k^\top\vec{e}_i\|_2^2}{k} \qquad 1 \leq i \leq n.
\end{equation}
It is easy to verify that $\sum_{i=1}^np_i^{\rm lev} = 1$, so the leverage scores define a discrete probability measure. Consider another probability measure that satisfies $p_i \geq \gamma p_i^{\rm lev} > 0$ for $1 \leq i \leq n$ and $\gamma \in (0,1]$, and construct the columns of the random matrix $\B\Omega$ as 
\[ \B\Omega = \bmat{ \frac{1}{\sqrt{\ell p_{t_1}}} \vec{e}_{t_1} & \dots & \frac{1}{\sqrt{\ell p_{t_\ell}}} \vec{e}_{t_\ell} }, \]
where $t_j $ are drawn uniformly at random from $\{1,\dots,n\}$ with replacement. Then it follows that
\[ \max_{1 \leq j \leq \ell }\|\mat{V}_k^\top \B\Omega\vec{e}_j\|_2 \leq K_k \equiv \sqrt{\frac{k}{\gamma}} \qquad   \max_{1 \leq j \leq \ell }\|\mat{V}_\perp^\top \B\Omega\vec{e}_j\|_2 \leq K_\perp \equiv \sqrt{\frac{ 1}{ \ell \min\limits_{1 \leq j \leq n}{p_j}}}.\] 
By \cref{thm:genindep}, the number of required samples is $\ell = \mc{O}(k\log k)$. However, computing the leverage scores is, in general, computationally challenging and there is a trade-off between the good features of leverage scores versus the computational cost in obtaining them; see~\cite[Section 9.6.4]{martinsson2020randomized}.

The main takeaway is that the number of samples has a logarithmic dependence on the target rank $k$. In general, this condition is necessary and is related to the coupon collector problem (see~\cite[Remark 11.2]{halko2011finding} and~\cite[Sections 5.5--5.6]{vershynin2010introduction}).

\subsection{Independent columns with bounded second moment}

\begin{definition}\label{def:randommat4}
In this model, we consider the random matrix $\mat\Omega \in \R^{n\times \ell}$ with the following properties:
   \begin{enumerate}
    \item Independence: the columns of the random matrix $\mat\Omega \in \R^{n \times \ell}$ are independent copies of a random vector.
    \item Normalization: each column is isotropic.
    \item Boundedness: the columns satisfy
    \[ \expect{\max_{1 \leq j \leq \ell}\|\mat{V}_k^\top\mat\Omega\vec{e}_j\|_2^2 } \leq K_M, \]
    where $K_M$ is allowed to depend on $n$ and $k$.    
\end{enumerate} 
\end{definition}

This is the most general class of random matrices we consider since it subsumes \cref{def:randommat1,def:randommat2,def:randommat3}. This is explained further in Section~\ref{ssec:otherex}. We now present a result for this class of random matrices. 

\begin{theorem}[Independent columns with bounded second moment]
\label{thm:indepinexp}
Let $\mat\Omega \in \R^{n\times \ell}$ be a random matrix satisfying~\Cref{def:randommat4}. 
Let $\varepsilon, \delta \in (0,1)$ be user-defined parameters. Suppose that the number of samples satisfies 
\[  n \geq \ell = \frac{K_MC_{BCS}}{\varepsilon^2 \delta^2 }\log(k).\]
Then, with a probability of failure at most $\delta$, we have 
\[ \|\mat\Sigma_\perp\Oh_2\Oh_1^\dagger\|_2  \leq \frac{\|\mat\Sigma_\perp\|_2}{(1-\varepsilon)\delta} (1 + 2\sqrt{\sr(\mat\Sigma_\perp)}). 
\]
If  $\mat\Omega$ is an input to~\cref{alg:randsvd}, then with a probability of failure at most $\delta$,
\[\|\mat{A} - \mathat{U}\mathat{\Sigma}\mathat{V}\t \|^2 \leq \|\mat\Sigma_\perp\|_2^2  + \frac{\gamma_k^{4q}\|\mat\Sigma_\perp\|_2^2 }{(1-\varepsilon)^2\delta^2} ( 1 + 2\sqrt{\sr(\mat\Sigma_\perp)})^2 . \]
Here, $C_{BCS}$ is an absolute constant.
\end{theorem}
\begin{proof}
See Section~\ref{subsec:bounded-col-exp}.
\end{proof}

By the discussion before the statement of~\Cref{thm:indepinexp}, we see that this result is applicable to a large class of random matrices. However, compared to the previous results, this has a worse dependence on the failure probability $\delta$. In the following discussion, we explore a class of random matrices that satisfies \cref{def:randommat4} and for which \cref{thm:indepinexp} therefore applies. We discuss a class of random matrices constructed using random variables whose tails can decay faster than sub-exponential.   In Supplementary Materials (Section~\ref{ssec:supp_mat4}), we also discuss two classes of distributions called uniform sampling from a convex set and log-concave distributions.

\paragraph{$\alpha$-sub-exponential}
A random variable $X$ is said to be $\alpha-$sub-exponential~\cite{Goetze_2021,sambale2022some} with parameter $0 < \alpha \leq 2$ if 
\[ \prob{ |X| \geq t} \leq c\exp(-Ct^\alpha), \qquad t \geq 0,\]
where $c,C$ are constants. Equivalently, this can be expressed in terms of the exponential Orlicz quasi-norm, i.e., $\|X\|_{\psi_\alpha} := \{\inf t > 0 \ | \ \expect{\exp(|X|^\alpha/t^\alpha)}\leq 2\}$.
Strictly speaking, it is only a norm for $\alpha \geq 1$. 
It generalizes the class of sub-Gaussian ($\alpha = 2$), sub-exponential ($\alpha =1$), sub-Weibull $\alpha \in (0,1]$ (that have heavy tails), and bounded ($\alpha > 0)$ random variables. We provide a result for the minimal number of samples required for this class of random matrices.
\begin{corollary}\label{cor:alpha}
 Let $\mat\Omega \in \R^{n \times \ell}$ be a matrix with independent copies of a random variable $X$ drawn from the $\alpha-$sub-exponential class with zero mean, unit variance and $\|X\|_{\psi_\alpha} \leq M$ with  $n \geq 3$ and  $M \geq 1$.  
    With the rest of the setup in \cref{thm:indepinexp}, the minimal number of samples for this model satisfies $$\ell = \mc{O}((k + M^2 (\sqrt{k\log n} + (2/\alpha)^{2/\alpha}))\log k).$$ 
\end{corollary}
\begin{proof}
    See Supplementary materials (Section~\ref{ssec:supp_mat4}).
\end{proof}
This result extends the analysis to a wide class of distributions. The sub-exponential family includes the exponential, Laplace, Logistic, chi-squared, and Poisson distributions. However, to use \cref{thm:indepinexp}, we have to center them appropriately to ensure zero mean and unit variance.

\subsection{Summary and applications of the results}

A summary of the number of samples required is given in \cref{tab:summary} for the {four} different models of random matrices discussed in Section~\ref{sec:main}. In the Supplementary Materials (Section~\ref{ssec:otherex}), we discuss a few additional examples and provide more details on some computations.

\begin{table}[!ht]
    \centering
    \begin{tabular}{c|c|c}\hline
    \multicolumn{3}{c}{Independent sub-Gaussian entries}\\ \hline
    Random matrix &  Parameters &  Num. samples $\ell$\\ \hline
    Generic & $K_E$ & $\mc{O}(K_E^4 k)$ \\
    Gaussian     &  - & $\mc{O}(k)$ \\
    Sparse Rademacher     &  Sparsity $s \in [1,\infty)$ &  $\mc{O}(s^2 k)$ \\
    Uniform     &  -  & $\mc{O}(k)$ \\
    \hline    
    \multicolumn{3}{c}{Independent sub-Gaussian columns}\\ \hline
    Generic & $K_C$ & $\mc{O}(K_C^4k)$ \\
    Spherical     &  -  & $\mc{O}(k)$ \\
    Sparse sign     & Nonzeros per column $N$ & $\mc{O}(k(n/N)^2)$  \\\hline
        \multicolumn{3}{c}{Independent bounded columns}\\ \hline
        Generic & $K_k$ & $\mc{O}(K_k^2 \log k)$\\
        Coordinate distribution & Coherence $\mu(\mat{V}_k)$ & $\mc{O}(n \mu(\mat{V}_k)\log k) $   \\
    Leverage score    & -   & $\mc{O}(k\log k)$  \\ \hline
    \multicolumn{3}{c}{Independent columns with bounded moments}\\ \hline 
    Generic & $K_M$ & $\mc{O}(K_M \log k)$\\
    $\alpha-$sub-exponential & & $\mc{O}( (k + \sqrt{k\log n})\log k)$ \\
    Log-concave & & $\mc{O}( (k+\log(n))\log(k) ) $ 
    \end{tabular}
    \caption{Summary of the minimal number of samples $\ell$ required for the different random matrix models. See also Supplementary Materials (Section \ref{ssec:otherex}). The parameters $K_C, K_k, $ and $K_M$ may implicitly depend on $n$ and $k$.}
    \label{tab:summary}
\end{table}

In this section, we have addressed the error in the low-rank approximation using the randomized subspace iteration method. In the Supplementary Materials (Section~\ref{sm:Nystrom}), we show how this analysis can be adapted for the Nystr\"om method. This shows that our bounds are applicable in a wider context in randomized low-rank approximations.

In some instances, it is required to bound $\uninorm{\B\Sigma_\perp \Oh_2\Oh_1^\dagger}$ in a unitarily invariant norm. To apply our results, we can use strong submultiplicativity~\cite[(IV.40) and Proposition IV.2.4]{bhatia1997matrix} to bound 
\[ \uninorm{\mat\Sigma_\perp \Oh_2\Oh_1^\dagger} \leq \|\mat\Sigma_\perp \Oh_2\Oh_1^\dagger\|_2 \uninorm{\mat{I}_k}. \] 
Using this approach, we incur a small additional factor $\uninorm{\mat{I}_k}$. However, this factor only depends on the target rank $k$, {and not on either of} the dimensions of the matrix $m$ and $n$. For the spectral norm, this factor is $1$, for the Frobenius norm, it is $\sqrt{k}$, whereas for the Schatten-1 norm, it is $k$.  

\section{Proofs}\label{sec:proofs}

For self-contained presentation, this section contains the proofs and details of the main results. Some of the proofs and auxiliary results are given in Supplementary Materials (Section~\ref{ssec:exproofs}).

\subsection{Auxiliary results}\label{ssec:aux}

\begin{proof}[Proof of \cref{lem:gauss}]
From the proof of~\cite[Theorem 10.8]{halko2011finding}, with probability at least $1-2t^{-p} -e^{-u^2/2}$,
\[\|\mat\Sigma_\perp\Oh_2\Oh_1^\dagger\|_2 \leq \|\mat\Sigma_\perp\|_2\sqrt{\frac{3k}{p+1}}\cdot t + \|\mat\Sigma_\perp\|_F\frac{e\sqrt{k+p}}{p+1}\cdot t +\|\mat\Sigma_\perp\|_2 \frac{e\sqrt{k+p}}{p+1}\cdot ut.  \]
Set $2t^{-p} = \delta/2$ and $e^{-u^2/2} = \delta/2$, and solve for $t$ and $u$. Plug these values into the above inequality.
\end{proof}

\begin{proof}[Proof of \cref{lem:samplecov}]
By submultiplicativity $\|\mat\Sigma_\perp\mat\Oh_2\Oh_1^\dagger\|_2 \leq \|\mat{Z}\|_2 \|\Oh_1^\dagger\|_2$, so we focus on each term separately. 
Consider $\|\mat{Z}\|_2$. Using the triangle inequality and the property $\|\mat{Z}\|_2^2= \|\mat{ZZ}\t\|_2$, we get 
\[ \|\mat{Z}\|_2^2  \leq \ell \left(\|\frac{1}{\ell}\mat{ZZ}\t- \mat{\Sigma}_\perp\mat\Sigma_\perp\t\|_2 + \|\mat\Sigma_\perp\|_2^2\right) \leq \ell\|\mat\Sigma_\perp\|_2^2 (1+\eta).  \]
Next, consider $\|\Oh_1^\dagger\|_2$. From the assumption in~\eqref{eqn:covestimates}, 
\[ |\frac{1}{\ell}\sigma_k^2(\Oh_1) -1| \leq \|\frac{1}{\ell}\Oh_1\Oh_1\t - \mat{I}\|_2 \leq \varepsilon < 1.\]
Therefore, $\ell(1 - \varepsilon) \leq \sigma_k^2 (\Oh_1) \leq \ell(1+\varepsilon)$. Since $\varepsilon < 1$, this gives $$\|\Oh_1^\dagger\|_2^2 = 1/\sigma_k^2(\Oh_1) \leq \frac{\ell^{-1}}{1 - \|\frac{1}{\ell}\Oh_1\Oh_1\t - \mat{I}\|_2 }\leq   \frac{1}{{(1-\varepsilon)\ell}}.$$ Combining intermediate results yields the desired bound. 
\end{proof}

\subsection{Independent sub-Gaussian entries}
\label{subsec:sub-Gaussian}
To establish the result of Theorem~\ref{thm:subgauss}, we need to prove the following result on the norm of a sub-Gaussian random matrix, which may be of independent interest.

\begin{theorem}\label{thm:subgaussnorm}
Let $\B{S} \in \R^{m\times n}$ be a random matrix with i.i.d sub-Gaussian entries with zero mean and sub-Gaussian norm $K := \max\limits_{ij} \|s_{ij}\|_{\psi_2}$. Then for any $\B{M} \in \R^{M\times m}$ and $\B{N}\in R^{n\times N}$ 
\begin{equation}\label{eqn:subgaussexpect}
    \expect{\|\B{MSN} \|_2} \leq C_{SNE} K \|\B{M}\|_2 \|\B{N}\|_2   (  \sqrt{\sr(\B{M})}    +  \sqrt{\sr(\B{N})} ).
\end{equation}
Furthermore, for every $u \geq 0 $
\begin{equation}\label{eqn:subgausstail} \|\B{MSN} \|_2 \leq C_{SNT}K\|\B{M}\|_2 \|\B{N}\|_2  ( \sqrt{\sr(\B{M})} + \sqrt{\sr(\B{N})} + u   ),
\end{equation}
with probability at least $1- 2\exp(-u^2)$. Here $C_{SNE}$ and $C_{SNT}$ are absolute constants.
\end{theorem}

The proof of this theorem is an application of Talagrand's comparison inequality which we recall here and is a special case of the sub-Gaussian Chevet theorem~\cite[Theorem 8.7.1]{vershynin2018high}. However, in the Supplementary Materials (Section~\ref{ssec:rmbounds}) we give some additional details and track the constants from Talagrand's comparison inequality. We will also need the following result which is a two-sided bound on the singular values of sub-Gaussian random matrices.
\begin{theorem}\label{thm:twosided}
    Let $\mat{S} \in \R^{m\times n}$ whose rows $\vec{e}_i^\top\mat{S}$ are independent, mean zero, isotropic, sub-Gaussian random vectors in $\R^n$ with $\max_{1 \leq i \leq m}\|\mat{S}^\top\vec{e}_i \|_{\psi_2} \leq K$. Then, for any $t\geq 0$, we have 
    \[ \sqrt{m} - C_{T}K^2(\sqrt{n} + t) \leq \sigma_n(\mat{S}) \leq \sigma_1(\mat{S}) \leq \sqrt{m} - C_{T}K^2(\sqrt{n} + t).\]
\end{theorem}
\begin{proof}
    See~\cite[Theorem 4.6.1.]{vershynin2010introduction}.
\end{proof}

With the results of~\Cref{thm:subgaussnorm} at hand, we are ready to prove~\Cref{thm:subgauss}.
\begin{proof}[Proof of~\cref{thm:subgauss}]
We will carry out the proof in three steps:
\paragraph{Step 1. {Bounding} $\|\Oh_1^\dagger\|_2$} Since the columns of $\Oh_1 =  \mat{V}_k^\top \mat\Omega \in \R^{k \times \ell}$ are orthonormal, they
are also independent isotropic random vectors. Furthermore, by~\Cref{lem:subgauss} their sub-Gaussian norm satisfy $\|\Oh_1\vec{e}_j\|_{\Psi_2} \leq \sqrt{C_{R}}K_E$ \ for \ $1\leq j \leq \ell$. Applying~\Cref{thm:twosided} to matrix $\mat\Omega_1^\top$ gives
\[ \prob{ \sqrt{\ell} - C_{T}C_RK_E^2(\sqrt{k} +t  ) \leq \sigma_k(\Oh_1) }  \leq 2 \exp(-t^2).   \]
Setting $\delta/2 = 2 \exp(-t^2)$ yields $t = V_\delta = \sqrt{\log\frac{4}{\delta}}$. Hence, with a probability of failure {at most $\delta/2$}
\[ \sigma_k(\Oh_1) \geq \sqrt{\ell} -C_{ES}K_E^2(\sqrt{k} +V_\delta ), \] 
with $C_{ES} := C_TC_R$.
Now, take  $\sqrt{\ell} = C_{ES}K_E^2 (\sqrt{k} + {V_\delta})/ \varepsilon = \mc{O}(K_E^2 \sqrt{k})$. Then, with probability at {least $1-\delta/2$}, $\sigma_k(\mathat{\Omega}_1) > 0$ and $$\|\Oh_1^\dagger\|_2 = \frac{1}{\sigma_k(\Oh_1)} \leq  \frac{1}{ (\varepsilon^{-1} - 1)C_{ES}K_E^2 (\sqrt{k} + {V_\delta})}  = \frac{1}{(1-\varepsilon) \sqrt{\ell}}, $$ which gives \begin{equation}\label{eqn:subgauss1} \prob{\|\Oh_1^\dagger\|_2 > \frac{1}{(1-\varepsilon)\sqrt{\ell}}}\leq {\frac{\delta}{2}}. \end{equation} 
 \paragraph{Step 2. {Bounding} $\|\mat\Sigma_\perp \Oh_2\Oh_1^\dagger\|_2$} 
 {Let $E$ denote the event that $\|\Oh_1^\dagger\|_2 \leq \frac{1}{(1-\varepsilon)\sqrt{\ell}}$; from Step 1, $\prob{E^c } \leq \delta/2$. Condition on the event $E$ and apply~\Cref{thm:subgaussnorm}} with $\B{M} = \mat{Z} \mat\Omega $ and $\B{N} = \Oh_1^\dagger$, where $\mat{Z} = \mat\Sigma_\perp \mat{V}_\perp^\top$. Note that this theorem applies since $\mat\Omega$ has i.i.d mean zero sub-Gaussian random variables and $\mat{Z\Omega} \Oh_1^\dagger = \mat\Sigma_\perp \Oh_2\Oh_1^\dagger$.  We get 
 
 \[ \prob{  \|\B\Sigma_\perp\Oh_2 \Oh_1^\dagger \|_2 > C_{SNE}K_E \|\B{Z}\|_2 \|\Oh_1^\dagger\|_2  ( \sqrt{\sr(\B{Z})} + \sqrt{\sr(\Oh_1^\dagger)} + u  \ )  \ | \ E } \leq 2e^{-u^2}. \] 
 Using the deterministic bound $ \|\Oh_1^\dagger\|_F \leq \sqrt{k} \|\Oh_1^\dagger\|_2$, and the bound on $\|\Oh_1^\dagger\|_2$ from \textit{Step 1} under the event $E$, we get 
 \[ \prob{  \|\B\Sigma_\perp\B{\Oh}_2 \B{ \Oh}_1^\dagger \|_2 > C_{SNE}K_E\frac{\|\B{Z}\|_2}{(1-\varepsilon)\sqrt\ell}  ( \sqrt{\sr(\B{Z})} + \sqrt{k} + u  ) \ | \ E   } \leq 2e^{-u^2}. \]
 By the submultiplicativity of $\|\cdot\|_\xi$, where $\xi \in \{2,F\}$ and the properties of right singular vectors $\mat{V}_\perp$, we obtain  $\| \B{Z}\|_{\xi} \leq   \|\B\Sigma_\perp\|_{\xi}$. Thus,
  \begin{equation}\label{eqn:subgauss2} \prob{  \|\B\Sigma_\perp \Oh_2 \Oh_1^\dagger \|_2 > C_{SNE}K_E   \frac{\|\B\Sigma_\perp\|_2}{(1-\varepsilon)\sqrt\ell} ( \sqrt{\sr(\B\Sigma_\perp)} + \sqrt{k} + u  ) \ | \ E } \leq 2e^{-u^2}. \end{equation}
 \paragraph{Step 3. Combining intermediate steps} Setting $\delta/2 = 2e^{-u^2}$, solving for $u$, substituting into~\eqref{eqn:subgauss2}, and removing the conditioning using~\eqref{eqn:subgauss1}    completes the proof for $\|\mat\Sigma_\perp\Oh_2\Oh_1^\dagger\|_2$ {with $C_{ES} :=C_TC_R$ \ and \ $C_{EB}:= C_{SNE}$}. The use of~\cref{thm:StrBound} completes the proof.
\end{proof}

\subsection{Independent sub-Gaussian columns}
\label{subsec:sub-Gaussian-col}

The next theorem only requires the columns of the sub-Gaussian random matrix to be independent, i.e., it doesn't require the individual matrix entries to be independent.

\begin{theorem}\label{thm:subgaussnorm2}
    Let $\B{S} \in \R^{m\times n}$ be a random matrix with independent sub-Gaussian isotropic columns that have sub-Gaussian norm $K:= \max\limits_{1 \leq j \leq \ell } \|\B{Se}_j\|_{\psi_2}$. Then for any $\B{M} \in \R^{M\times m}$ 
\begin{equation}\label{eqn:subgaussexpect2}
      \expect{\|\B{MS}\|_2} \leq \sqrt{n}\|\B{M}\|_2 +      C_{MDE}K^2 { \|\B{M}\|_F}    .
\end{equation}
Furthermore, for every $u \geq 0 $
\begin{equation}\label{eqn:subgausstail2} \|\B{MS}\|_2 \leq \sqrt{n} \|\B{M}\|_2 +  C_{MDT}K^2 [{\|\B{M}\|_F} + u \|\B{M}\|_2],     \end{equation}
with probability at least $1- 2\exp(-u^2)$.
Here {$C_{MDE}, C_{MDT}$} are absolute constants defined in \cref{thm:matdev}.
\end{theorem}
Note that if $\B{M} = \B{I}$ we recover the results of~\cite[Theorem 4.6.1]{vershynin2018high} and \cite[Exercise 4.6.3]{vershynin2018high}. The proof relies on the matrix deviation inequality and is relegated to the Supplementary Materials (Section~\ref{ssec:rmbounds}).

Now, analogous to the proof of~\Cref{thm:subgauss}, we can conduct
the proof of~\Cref{thm:subgauss2}.

\begin{proof}[Proof of \cref{thm:subgauss2}]
The proof is similar to the proof of \cref{thm:subgauss}, but it uses the results of \cref{thm:subgaussnorm2} instead of \cref{thm:subgaussnorm}. First, we apply submultiplicativity to obtain \begin{equation}\label{eqn:subgaussinter2}\|\mat\Sigma_\perp \Oh_2\Oh_1^{\dagger}\|_2 \leq \|\mat\Sigma_\perp \Oh_2\|_2 \|\Oh_1^\dagger\|_2  .\end{equation}
The bound for the low-rank approximation follows from the bound for $\|\mat\Sigma_\perp \Oh_2\Oh_1^{\dagger}\|_2$.
\paragraph{Step 1: Bounding $\|\Oh_1^\dagger\|_2$} The columns of $\B\Omega \in \R^{n \times \ell}$ are independent sub-Gaussian with norm $\|\mat\Omega\vec{e}_j\|_{\psi_2}\leq K_C$. Thus, by \cref{lem:subgauss2}, the columns of $\Oh_1$ are independent, isotropic, and sub-Gaussian with norm $\|\Oh_1 \vec{e}_j\|_{\psi_2} \leq \sqrt{C_R}K_C $ for $1 \leq j \leq \ell$. Analogously to the proof of \cref{thm:subgauss}, with $\ell = C_{CS} C_RK_C^2\varepsilon^{-2} (\sqrt{k}+ V_\delta)^2$,  we once again obtain~\eqref{eqn:subgauss1}. Here $C_{CS} := C_TC_R$, where $C_T$ is the constant in \Cref{thm:twosided} and $C_R$ is the constant in \cref{lem:subgauss2}. 
\paragraph{Step 2: Bounding $\|\mat\Sigma_\perp \Oh_2\|_2$} By a similar argument, the columns of $\Oh_2$ are independent, isotropic, sub-Gaussian with sub-Gaussian norm at most $C_RK_C$. Applying \cref{thm:subgaussnorm2} with $\B{M} = \B\Sigma_\perp$, and setting $2\exp(-u^2) = \delta/2$ yields
\begin{equation}\label{eqn:subgauss2alt} \prob{\|\mat\Sigma_\perp \Oh_2\|_2 >  \|\mat\Sigma_\perp\|_2  \left( \sqrt{\ell}  + C_{MDT}C_RK_C^2  (\sqrt{\sr(\mat\Sigma_\perp)}  + V_\delta) \right) } \leq \delta/2. \end{equation}
\paragraph{Step 3: Combining the steps} Combining~\eqref{eqn:subgauss2alt} and~\eqref{eqn:subgauss1} using the union bound, and inserting into~\eqref{eqn:subgaussinter2} completes the proof with {$C_{CB} := C_{MDT}C_R$}.
\end{proof}

\subsection{Independent bounded columns}
\label{subsec:bounded-col}

\begin{proof}[Proof of \cref{thm:genindep}]
It suffices to bound $\|\mat\Sigma_\perp\Oh_2\Oh_1^\dagger\|_2.$ The proof proceeds as before in three steps. \paragraph{Step 1: Bounding $\|\Oh_1^\dagger\|_2$}
Define the matrices $\mat{X}_j = (\mat{V}_k^\top\mat\Omega\vec{e}_j) (\mat{V}_k^\top\mat\Omega\vec{e}_j)^\top$ for $1 \leq j \leq \ell$. Then $\mat{X}_j$ are positive semidefinite with $\lambda_{\max}(\mat{X}_j) = \|\mat{V}_k^\top\mat\Omega\vec{e}_j\|_2^2 \leq K_k^2$ \ and \ $\expect{\mat{X}_j} = \mat{I}_k$. Therefore, $\lambda_{\min}(\expect{\sum_{j=1}^\ell \mat{X}_j}) = \lambda_{\min}(\ell \mat{I}) = \ell$. By the matrix Chernoff inequality~\cite[Theorem 5.1.1]{tropp2015introduction} for some parameter $\varepsilon\in {(0,1)}$
\begin{equation}
\label{eq:GICChernoff}
\prob{\lambda_{\min}(\sum_{j=1}^\ell \mat{X}_j) \leq (1-\varepsilon)\ell} \leq k f(-\varepsilon)^{\ell/K_k},
\end{equation}
where $f(\varepsilon) = e^{\varepsilon}/(1+\varepsilon)^{(1+\varepsilon)}$. Since $\sum\limits_{j=1}^\ell \mat{X}_j = \Oh_1\Oh_1^\top$ and ${f(-\varepsilon)}\leq e^{-\varepsilon^2/2}$, the Chernoff bound \eqref{eq:GICChernoff} takes the form 
$ \prob{\|\Oh_1^\dagger\|_2^2 \geq \frac{1}{(1-\varepsilon)\ell}} \leq ke^{-\varepsilon^2 \ell/(2K_k^2)}$. 
Setting $\delta/2 = ke^{-\varepsilon^2 \ell/(2K_k^2)}$ and solving for $\ell$ yields $\ell = 2K_k^2\log(2k/\delta)/\varepsilon^2$. This gives the first probabilistic bound
\begin{equation}\label{eqn:genbound1}
\prob{\|\Oh_1^\dagger\|_2 \geq \frac{1}{\sqrt{(1-\varepsilon)\ell}}} \leq \delta/2.
\end{equation}
\paragraph{Step 2: Bounding $\|\mat\Sigma_\perp \Oh_2\|_2$} Write $\mat{Z} = \mat\Sigma_\perp \Oh_2$, so that its columns take the form $\vec{z}_j = \mat\Sigma_\perp \mat{V}_\perp^\top (\mat\Omega\vec{e}_j) $ for $1\leq j \leq \ell $ and are independent random vectors. Define the matrices
 \[ \mat{Y} :=  \sum_{j  = 1}^\ell \mat{X}_j, \quad \mbox{ with } \quad \mat{X}_j := \frac{1}{\ell}(\vec{z}_j\vec{z}_j^\top - \mat\Sigma_\perp \mat\Sigma_\perp^\top).\]
 Therefore, $\mat{Y} = \frac{1}{\ell}\mat Z \mat Z^\top - \mat \Sigma_\perp \mat \Sigma_\perp^\top$ is a sum of independent random symmetric matrices. Furthermore, the matrices $\mat{X}_j$ have zero mean, so by the linearity of expectations, $\mat{Y}$ also has zero mean. Next, by the triangle inequality and the assumption on $\mat\Omega$ and $K_\perp$
 \[ \|\mat{X}_j\|_2 \leq \frac{1}{\ell} \|\mat\Sigma_\perp\|_2^2 (K_\perp^2 + 1) \leq \frac{2K_\perp^2}{\ell} \|\mat\Sigma_\perp\|_2^2 =: L  .  \]
 For the variance term, once again by the linearity of expectation, we obtain
 \[ \begin{aligned} \text{Var}[\mat{Y}] = & \> \sum_{j=1}^\ell \expect{\mat{X}_j^2 } =   \frac{1}{\ell^2} \sum_{j=1}^\ell  \expect{ (\vec{z}_j\vec{z}_j^\top - \mat\Sigma_\perp \mat\Sigma_\perp^\top)^2 } \\
 = & \> \frac{1}{\ell^2} \sum_{j=1}^\ell  \left[ \expect{ (\vec{z}_j\vec{z}_j^\top)^2 } - (\mat\Sigma_\perp \mat\Sigma_\perp^\top )^2 \right] \\
 \preceq & \> \frac{1}{\ell^2} \sum_{j=1}^\ell   \mat\Sigma_\perp\mat\Sigma_\perp\t \left[ K_\perp^2\|\mat\Sigma_\perp\|_2^2  \mat{I}   - \mat\Sigma_\perp\mat\Sigma_\perp^\top \right] \preceq \frac{K_\perp^2\|\mat\Sigma_\perp\|_2^2 }{\ell} \mat\Sigma_\perp\mat\Sigma_\perp^\top = :\mat{V}.
 \end{aligned}\]
 Note that in the last step we have used $(\vec{z}_j\vec{z}_j\t)^2 \preceq \|\vec{z}_j\|_2^2 \mat\Sigma_\perp\mat\Sigma_\perp\t \preceq K_\perp^2 \|\mat\Sigma_\perp\|_2^2 \vec{z}_j\vec{z}_j\t$. 
 The matrix $\mat{V}$ satisfies $ \|\mat{V}\|_2 = {K_\perp^2\|\mat\Sigma_\perp\|_2^4}/{\ell} =:  \nu$ and $$d := \intdim(\mat{V}) = \trace(\mat\Sigma_\perp\mat\Sigma_\perp\t)/ \|\mat\Sigma_\perp\|_2^2 = \sr(\mat\Sigma_\perp).$$ Therefore, by matrix Bernstein inequality~\cite[Theorem 7.7.1] {tropp2015introduction} for $t \geq \sqrt{\nu} + L/3$
 \[ \prob{ \|\frac{1}{\ell}\mat{ZZ}^\top - \mat\Sigma_\perp^2 \|_2 \geq t  } \leq 8d\exp\left(-\frac{t^2}{\nu + Lt/3} \right).   \]
 Note that when using~\cite[Theorem 2.3]{hallman2022monte}, the only requirement on $t$ is that it is nonnegative. Then,
 setting $\delta/2 = 8d\exp\left(-\frac{t^2}{\nu + Lt/3} \right) $, and solve for $t$ to get $t = \frac{L\beta_\delta}{6} + \sqrt{\frac{L^2 \beta_\delta^2}{36} + \beta_\delta \nu}$ where $\beta_\delta =\log \frac{16d}{\delta}$. By triangle inequality, and completion of the squares, \[ \frac1\ell\|\mat{Z}\|_2^2 \leq \| \mat\Sigma_\perp\|_2^2 + t \leq  \| \mat\Sigma_\perp\|_2^2 + \frac{L\beta_\delta}{3} + \frac{3\nu}{L}= \|\mat\Sigma_\perp\|_2^2 \left(\frac{2K_\perp^2 \beta}{3\ell} + \frac{5}{2} \right).\]
 Therefore, $\|\mat\Sigma_\perp\Oh_2\|_2 \leq   \| \mat\Sigma_\perp\|_2(K_\perp^2 \beta + 3 \ell)^{1/2}$ with probability at least $1-\delta/2$.

 \paragraph{Step 3: Combining the bounds} Using the union bound to combine the results from Steps 1 and 2 completes the proof.
\end{proof}

\subsection{Independent columns with bounded moments} \label{subsec:bounded-col-exp}
Before we prove \cref{thm:indepinexp}, we require the following non-commutative matrix Khintchine inequality. We briefly recall the Schatten-p norm~\cite[Equation (IV.31)]{bhatia1997matrix} 
\[
\|\mat{A}\|_{(p)} =\left(\trace(\mat{A}\t\mat{A})^{p/2}\right)^{1/p} =\left( \sum_{j=1}^{\min\{m,n\}}\sigma_j(\mat{A})^p\right)^{1/p}
\]
for $p \in [1,\infty)$.
A special case mentioning is that the Schatten-2 norm is the Frobenius norm. 
\begin{theorem}[Non-commutative Khintchine inequality]\label{thm:khintchine}
    Let $\mat{A}_i$ for $1 \leq i \leq N$ be a sequence of symmetric deterministic matrices and let $\epsilon_1,\dots, \epsilon_N$ be independent Rademacher random variables. Then {for $p \in [1,\infty)$} 
    \[ \left(\expect{ \left\|\sum_{i=1}^N\epsilon_i \mat{A}_i \right\|_{(p)}^p}\right)^{1/p}\leq \sqrt{p}\left\|\left(\sum_{i=1}^N\mat{A}_i^2\right)^{1/2}\right\|_{(p)}.  \]
\end{theorem}
\begin{proof}
    See~\cite[Corollary 7.3]{mackey2014matrix} for a proof with explicit constants, but see references in that paper for earlier proofs. 
\end{proof}
\begin{proof}[Proof of \cref{thm:indepinexp}]
It suffices to bound $\|\mat\Omega_\perp\Oh_2\Oh_1^\dagger\|_2.$ First, we apply submultiplicativity
to obtain~\eqref{eqn:subgaussinter2}. 
The proof proceeds as before in three steps.
\paragraph{Step 1. Bounding $\|\Oh_1^\dagger\|_2$}

Since the rows of the matrix
$\Oh_1^\top$ are independent isotropic random vectors in $\R^k$ with 
the boundedness property in Definition~\ref{def:randommat4},  
by~\cite[Theorem 5.45]{vershynin2010introduction}
\[ \expect{\max_{j \leq k }|\sigma_j(\Oh_1) -\sqrt{\ell}|} \leq   C\sqrt{K_M\log k}. \]
Therefore, by Markov's inequality~\cite[Proposition 1.2.4]{vershynin2018high}, with failure probability at {most $\delta/2$}, we have  $|\sigma_k(\Oh_1) -\sqrt{\ell}| \leq   \frac{C}{\delta}\sqrt{K_M\log k}$, or equivalently\[ \sqrt{\ell} - \frac{2C}{\delta}\sqrt{K_M\log k} \leq \sigma_k(\Oh_1) \leq \sqrt{\ell} + \frac{2C}{\delta}\sqrt{K_M\log k}.  \] 
If we take $\sqrt{\ell} = 2C/(\varepsilon\delta)\sqrt{K_M\log k}$, then  with probability {at least $1-\delta/2$}, $\sigma_k(\mathat{\Omega}_1) > 0$ and
$$\|\Oh_1^\dagger\|_2 = \frac{1}{\sigma_k(\Oh_1)} \leq  \frac{1}{ (\varepsilon^{-1} - 1)2C/\delta\sqrt{K_M\log k}}  = \frac{1}{(1-\varepsilon) \sqrt{\ell}}, $$ 
which gives~\eqref{eqn:subgauss1}.
\paragraph{Step 2. Bounding $\|\mat\Sigma_\perp\Oh_2\|_2$}
 
Let $\mat{Z} = \mat\Sigma_\perp\mat\Oh_2$. Since $\expect{\mat{ZZ}\t} = \ell \mat\Sigma_\perp\mat\Sigma_\perp\t$, by the symmetrizing argument~\cite[Lemma 6.4.2 and Exercise 6.4.4]{vershynin2018high}
\[\|\frac1\ell\mat{ZZ}\t - \mat{\Sigma}_\perp\mat\Sigma_\perp^\top  \|_2 \leq \frac{2}{\ell} \expect{ \expect{\| \sum_{j=1}^\ell \epsilon_j \vec{z}_j\vec{z}_j\t \|_2 } }, \]
where $\epsilon_j$ are independent Rademacher random variables
(also independent of the columns $\vec{z}_j = \mat{Ze}_j$, $1 \leq j \leq \ell$). The inner expectation is with respect to the Rademacher random variables and the outer expectation is with respect to $\mat\Omega$. We tackle the inner expectation first. By Jensen's inequality and the properties of the  Schatten-p norms
\[\expect{\| \sum_{j=1}^\ell \varepsilon_j \vec{z}_j\vec{z}_j\t \|_2} \leq \left( \expect{\| \sum_{j=1}^\ell \varepsilon_j \vec{z}_j\vec{z}_j\t \|_2^p}\right)^{1/p} \leq \left( \expect{\| \sum_{j=1}^\ell \varepsilon_j \vec{z}_j\vec{z}_j\t \|_{(p)}^p}\right)^{1/p}.  \] 
Next, we apply the non-commutative matrix Khintchine inequality (\Cref{thm:khintchine}) with $p=2$, to get 
\[ \begin{aligned}\expect{\| \sum_{j=1}^\ell \varepsilon_j \vec{z}_j\vec{z}_j\t \|_2} \leq & \>  \sqrt{2}  \left\| \left( \sum_{j=1}^\ell  (\vec{z}_j\vec{z}_j\t)^2 \right)^{1/2}\right\|_{(2)} \\
= & \>  \sqrt{2} \left( \trace(\sum_{j=1}^\ell(\vec{z}_j\vec{z}_j\t)^2) \right)^{1/2} =  \> \sqrt{2} \left(\sum_{j=1}^\ell \|\vec{z}_j\|_2^4 \right)^{1/2}. \end{aligned}\] 
In the second step, we have used the fact the Schatten-2 norm is the Frobenius norm, and for a symmetric positive semidefinite matrix $\|\mat{M}^{1/2}\|_{(2)} = \|\mat{M}^{1/2}\|_F = \sqrt{\trace(\mat{M})}$. In the last step, we used the cyclic property of the trace operator.  Combining all the intermediate steps gives
\[ \expect{\|\frac1\ell\mat{ZZ}\t - \mat{\Sigma}_\perp\mat\Sigma_\perp\t  \|_2} \leq \frac{2}{\ell} \sqrt{2} \expect{\left( \sum_{j=1}^\ell \|\vec{z}_j\|_2^4\right)^{1/2}}.\]
Using the subadditivity of the square root  and the fact that $\expect{\|\vec{z}_j\|_2^2} = \trace(\mat\Sigma_\perp\mat\Sigma_\perp\t)$ yields 
\[ \expect{\|\frac1\ell\mat{ZZ}\t - \mat{\Sigma}_\perp\mat\Sigma_\perp\t  \|_2} \leq \frac{2}{\ell} \sqrt{2} \left( \sum_{j=1}^\ell\expect{\|\vec{z}_j\|_2^2 } \right) = (2\sqrt{2}) \|\mat\Sigma_\perp\|_F^2.  \]
By the triangle inequality, it is easy to show that 
\[\expect{\|\mat{Z}\|_2^2} \leq  \|\mat\Sigma_\perp\|_2^2 {\ell}( 1  + (2\sqrt{2})\sr(\mat\Sigma_\perp)).   \]
Once again using the subadditivity of the square root, along with Markov's inequality we get 
\[ \prob{\|\mat{Z}\|_2 \geq  \frac{2\sqrt{\ell}\|\mat\Sigma_\perp\|_2}{\delta}( 1 + 2\sqrt{  \sr(\mat\Sigma_\perp))}} \leq \delta/2.\]
\paragraph{Step 3}
Combining
the results of Steps 1 and 2 using the union bound complete the proof.
\end{proof}

\section{Numerical Experiments}\label{sec:num}

In this section, we present numerical experiments to highlight various features of our analysis and illustrate the bounds. As mentioned earlier, we focus on the randomized subspace iteration method with $q=1$, \cref{alg:randsvd}, to illustrate the bounds. In Supplementary Materials (Section~\ref{sm:Nystrom}), we consider the performance of truncated rank Nyström method~\cite{tropp2017fixed}
for a PSD approximation of symmetric positive semidefinite (PSD) matrix $\mat A$. We also present numerical results for other distributions in Supplementary Materials (Section~\ref{sec:othernum}).

\subsection{Setup for the numerical experiments}

\paragraph{Distributions}

In the numerical experiments below, we consider three sets of random matrix distributions. 
\begin{enumerate}
    \item \textbf{Independent entries}: The first set satisfies \cref{def:randommat1} and includes the Gaussian distribution, Rademacher, Sparse Rademacher (with sparsity parameter $s=10$), Uniform distribution in the interval $[-\sqrt{3},\sqrt{3}]$.
    \item  \textbf{Independent columns}: The second set includes the column spherical distribution which satisfies \cref{def:randommat2}. We also consider uniform sampling from the columns of the Hadamard matrix without replacement which satisfies \cref{def:randommat2,def:randommat3} and has similarities to the sub-sampled randomized Hadamard transform. The last two examples include uniform sampling from the $\ell_1$ ball (satisfying \cref{def:randommat4}), and the $\ell_2$ ball of radius $\sqrt{n}$ (satisfying \cref{def:randommat2}; see~\cite[Exercise 3.4.7]{vershynin2018high}). For the last two distributions, we use the approach in~\cite{calafiore1998uniform} to generate the random matrices. 
    \item \textbf{$\alpha-$sub-exponential}: In this set, we include some examples from $\alpha-$sub-exponential distribution which satisfies \cref{def:randommat4}. The random matrix is filled with independent copies of the random variable distributed according to the following distributions: (1) Laplace (with scale parameter $\lambda =1$),  (2) centered Poisson (with parameter $\lambda = 10$), (3) Logistic (with scale parameter $s=1$), and (4) centered Weibull (with parameters $a=1, b=0.5$). These random variables were generated using {\sc Matlab}'s `Statistics and Machine Learning Toolbox'. 
\end{enumerate}
Note that some of the random matrices thus constructed do not have isotropic columns. However, in each case, the covariance matrix corresponding to each column is a scaled identity matrix. This does not affect the analysis for the error in the low-rank approximation since the term $\|\mat\Sigma_\perp \Oh_2\Oh_1^\dagger\|_2$ is invariant to rescaling all the columns. In case the mean is not zero, as in the case of Poisson and Weibull distributions, we center them by subtracting the means $\lambda$ and $a\Gamma(1+1/b)$ respectively, to ensure zero means.

\subsection{Experiment 1: Different singular value distributions} \label{ssec:exp1}

\begin{figure}[!ht]
    \centering
    \includegraphics[scale=0.25]{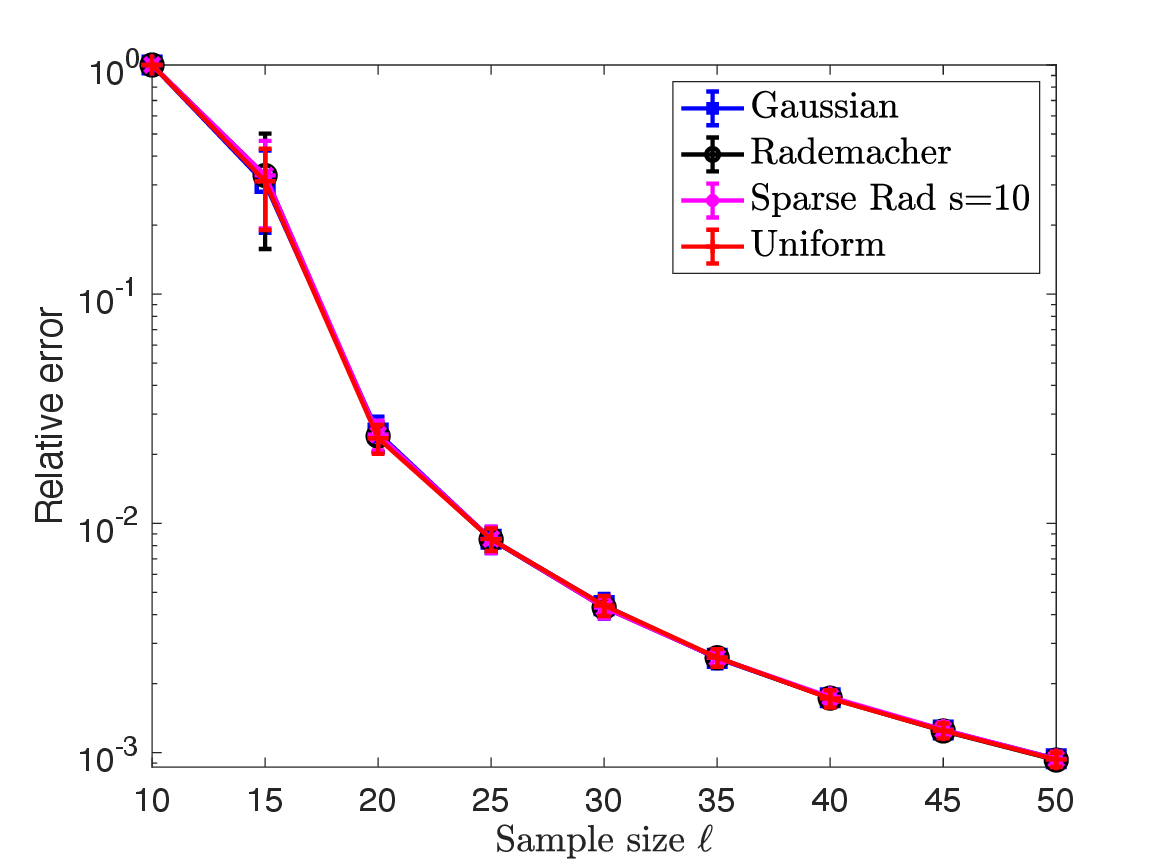}
    \includegraphics[scale=0.255]{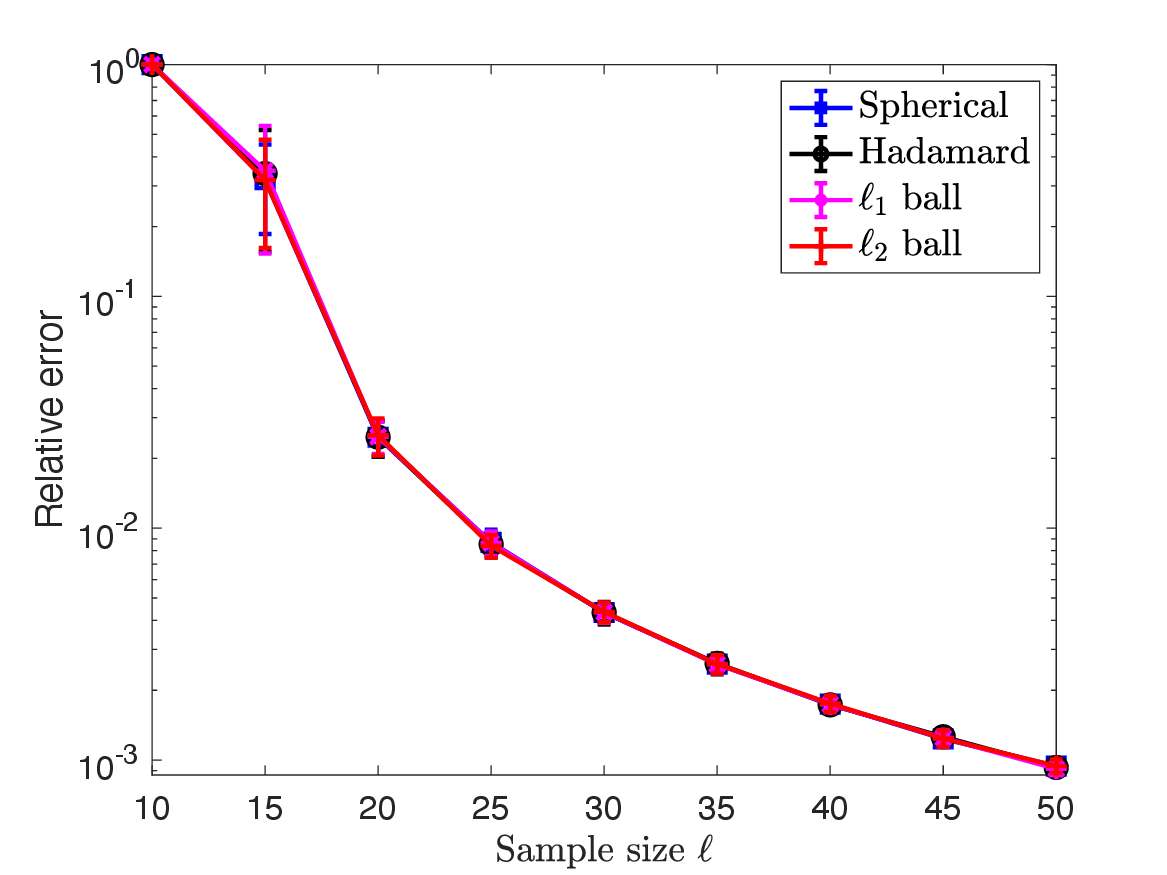}
    \includegraphics[scale=0.255]{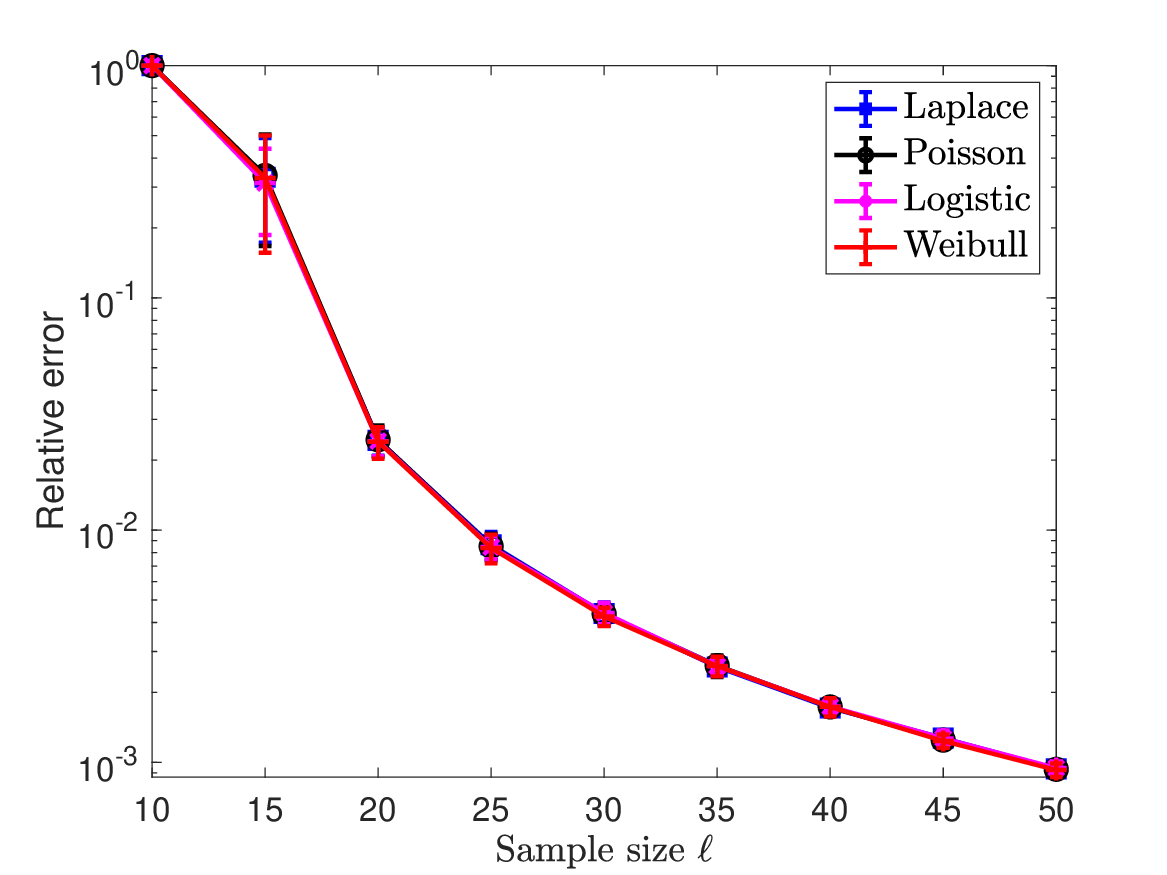}
    \caption{Relative error (RE) vs. sample size $\ell$ for randomized SVD low-rank approximation of the \texttt{FastDecay} test matrix for each set of random matrices.}
    \label{fig:dist_decay}
\end{figure}
In this experiment, we consider the effect of the random matrix $\mat\Omega \in \R^{n\times \ell}$ on the relative error. Consider two different synthetic test matrices $\mat{A}$ described below: 
\begin{enumerate}
    \item \texttt{FastDecay}: The matrix $\mat{A} \in \R^{n\times n}$ is constructed through its SVD $\mat{A} = \mat{U\Sigma V}\t.$ The matrices $\mat{U}$ and $\mat{V}$ are orthogonal matrices, first generated randomly and then computing the QR factorizations. The singular value matrix $\mat\Sigma$ is constructed as 
    \[ \mat\Sigma = \diag(\underbrace{1,\dots,1}_{r}, 2^{-d},3^{-d},\dots, (n-r+1)^{-d}), \]
    where $d=2$ is the degree of decay, $n=256$, and $r=15$.
    \item \texttt{ControlledGap:} The matrix $\mat A \in \R^{m \times n}, m \geq n$, is defined as 
    \[ \mat{A}  = \sum_{j=1}^r \frac{10}{j} \vec{x}_j\vec{y}_j\t + \sum_{j=r+1}^{\min\{m,n\}}\frac{1}{j}\vec{x}_j\vec{y}_j\t,  \]
    where $\vec{x}_j \in \R^{m}$ and $\vec{y}_j \in \R^{n}$ are sparse random vectors with density $0.25$ created using the {\sc Matlab} command \texttt{sprand}. We take $m=3000$ and $n=256$. 
    
\end{enumerate}

Throughout this section, the relative error is defined as 
\[\text{RE} = \frac{\|\mat{A}-\mat{A}_k\|_2}{\|\mat{A}\|_2 }. \]
We vary the number of samples $\ell$, compute the low-rank approximation using \cref{alg:randsvd} (over $100$ independent samples), and plot the average relative error versus the number of samples. The error bars in the plots correspond to one standard deviation of the relative error. The results for the \texttt{FastDecay} matrix are visualized in~\cref{fig:dist_decay} and for the \texttt{ControlledGap} in~\cref{fig:dist_control}. In each plot, the left panel corresponds to the test matrices of the Independent entries type, the middle one to the Independent column type, and  the right one to the $\alpha-$sub-exponential type.

\begin{figure}[!ht]
    \centering
    \includegraphics[scale=0.26]{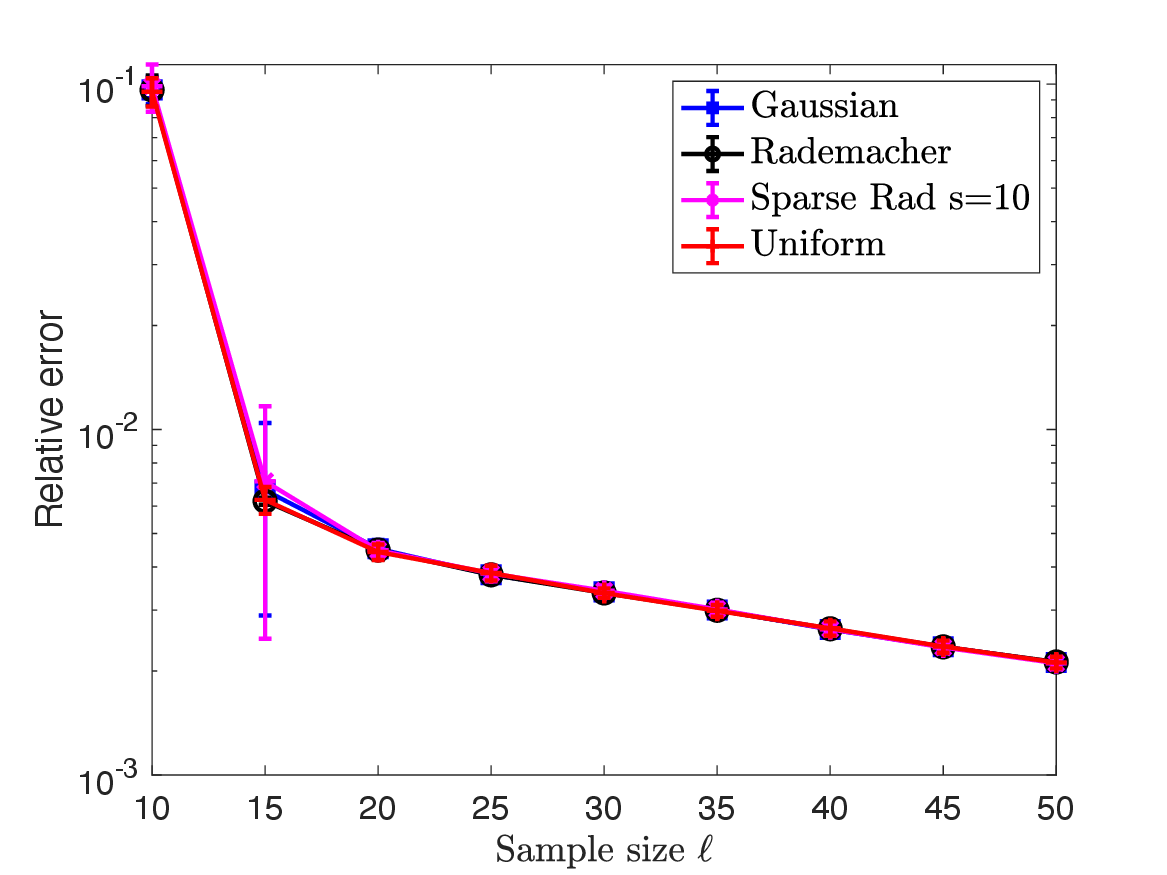}
    \includegraphics[scale=0.26]{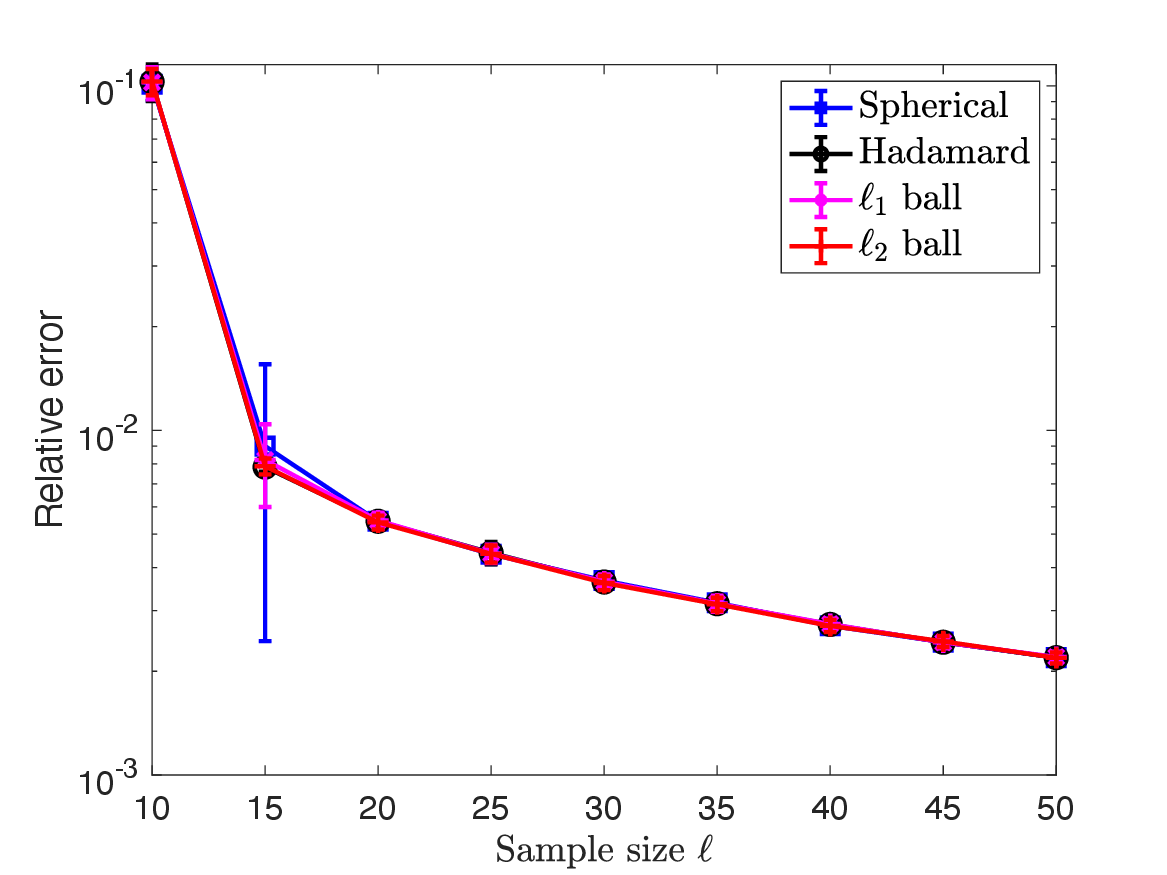}
    \includegraphics[scale=0.26]{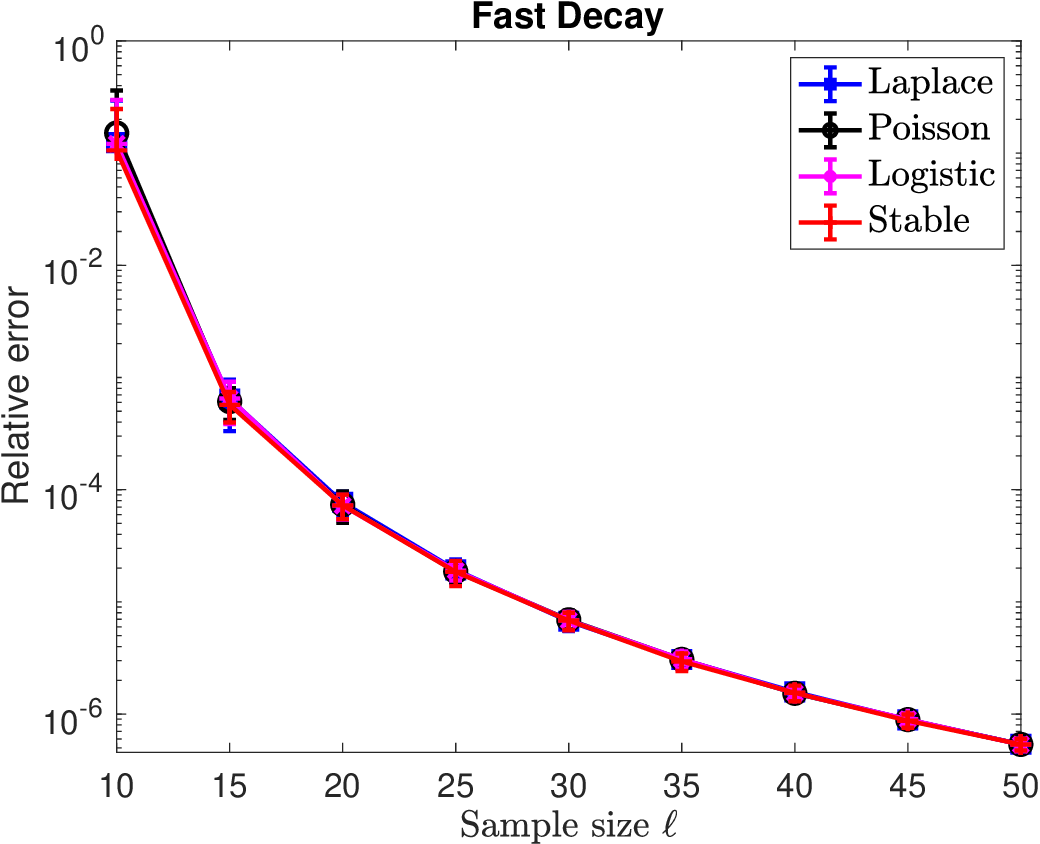}
    \caption{Relative error (RE) vs. sample size $\ell$ for randomized SVD low-rank approximation of the \texttt{ControlledGap} test matrix for each set of random matrices.}
    \label{fig:dist_control}
\end{figure}

For the \texttt{FastDecay} matrix, the relative error shows sharp decay with an increasing number of samples with a sharp transition at index $15$. All the distributions show very similar qualitative behavior with a large error bar at index $15$ which corresponds to the transition of the spectrum from flat to decaying. For this test matrix, this numerical observation is consistent with \cref{thm:subgauss} and \cref{thm:subgauss2} for matrices satisfying \cref{def:randommat1} and \cref{def:randommat2}.  For the \texttt{ControlledGap} matrix, the error decays sharply before index $15$ after which the decay is slower. All distributions have very similar behavior and the error bar appears largest around index $15$. {Furthermore, the presented results provide some numerical evidence that even heavy-tailed distributions can be used for low-rank approximations.}

\subsection{Experiment 2: Real-world dataset}
\label{sec:Exp2}

In this example, we consider the ocean surface temperature data set which is publicly available as NOAA Optimum Interpolation (OI) SST V2~\cite{noaa_oi_sst_v2}. This data set is in the form of a time series in which a snapshot is recorded every week in the span of 1990–2016 and data are available at a resolution of $1^\circ \times 1^\circ$. In total, there are $1713$ snapshots representing the columns of the matrix $\mat{A}$, which gives the \texttt{SeaSurfaceTemp} test matrix of size $64800\times 1713$. {We again demonstrate the effect of different random matrices $\mat{\Omega} \in \R^{n \times \ell}$ on the quality of the low-rank approximations, by varying the number of samples $\ell$ and plotting the average relative error as well as its standard deviation. The results for all three sets of random matrices are presented in \Cref{fig:sst}. The relative error admits very slow decay with an increasing number of samples, but the performance is comparable across distributions.} 

\begin{figure}[!ht]
    \centering
    \includegraphics[scale=0.26]{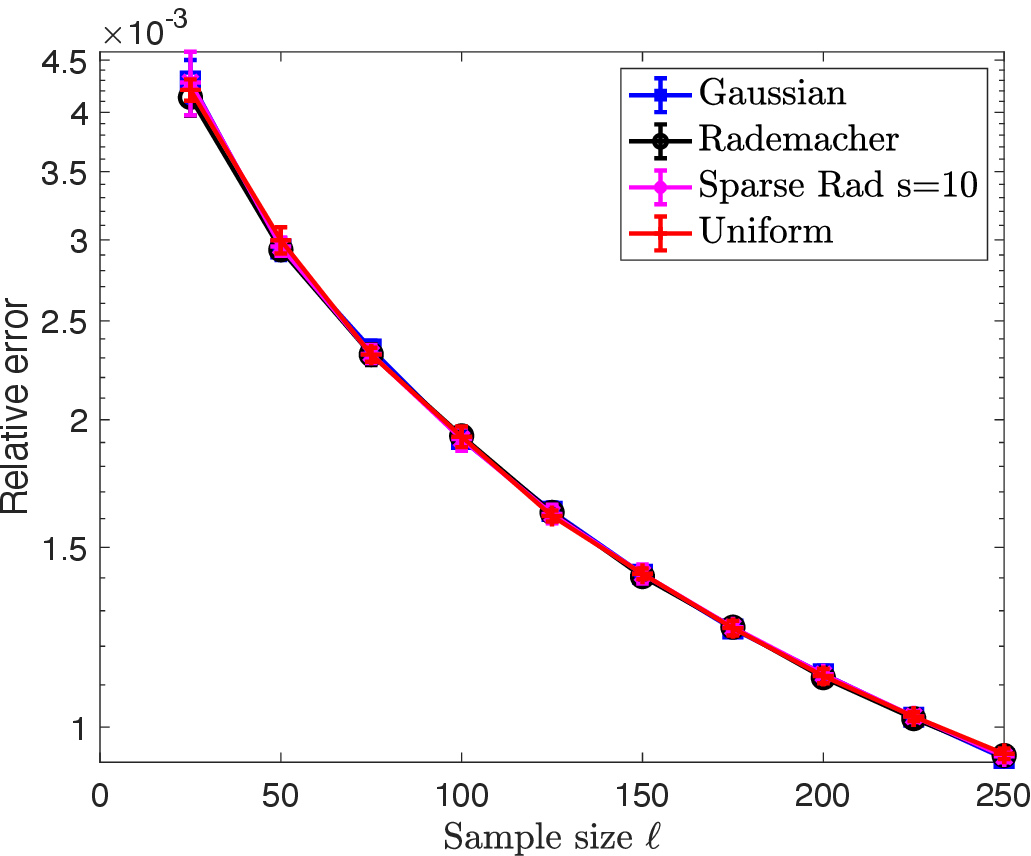}
    \includegraphics[scale=0.26]{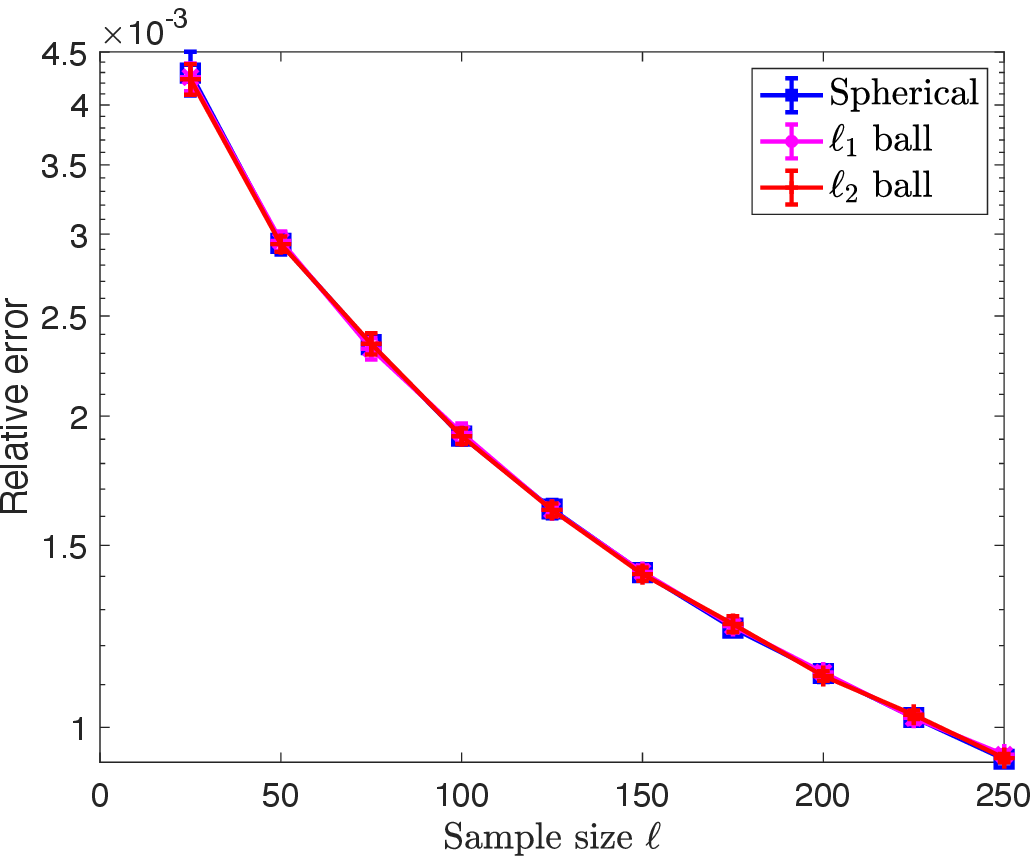}
    \includegraphics[scale=0.26]{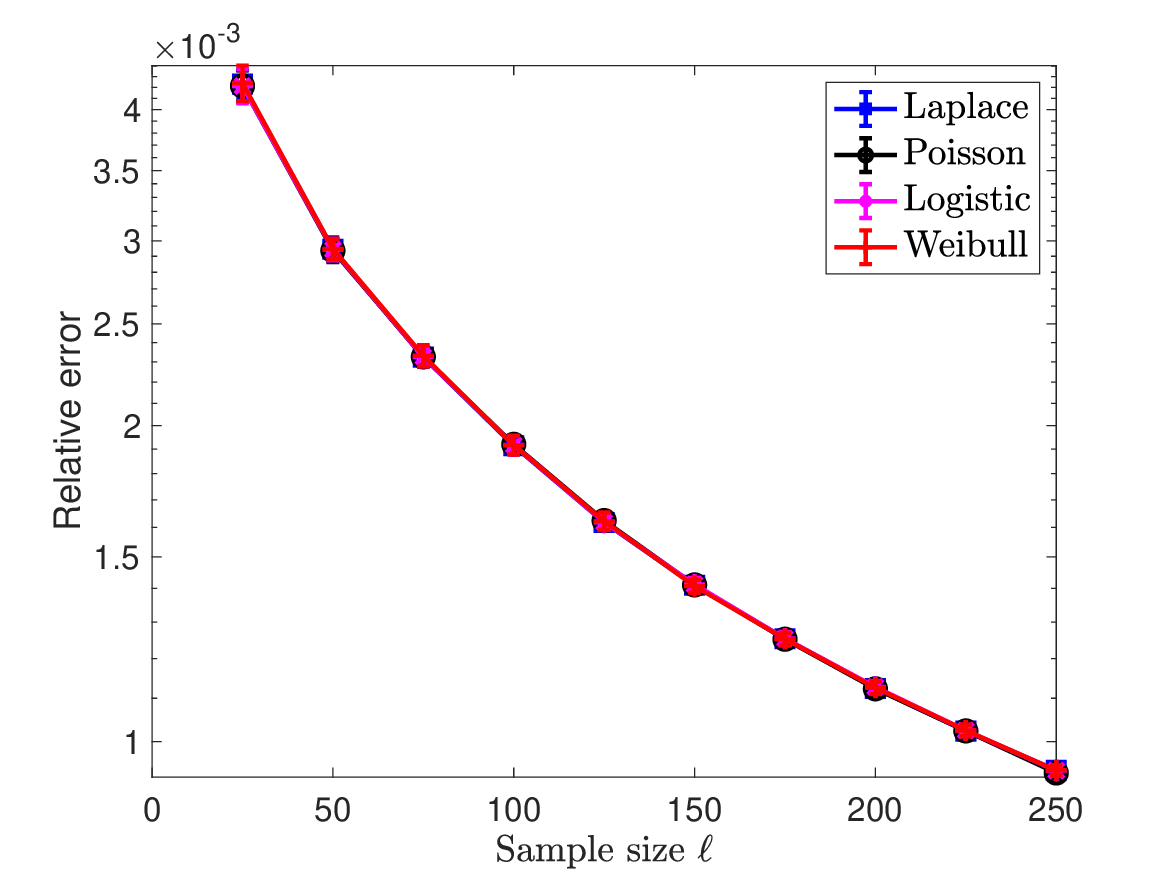}
    \caption{Relative error (RE) vs. sample size $\ell$ for the \texttt{SeaSurfaceTemp} test matrix for each set of random matrices.}
    \label{fig:sst}
\end{figure}

\section{Conclusions} In this paper, we provide analysis for randomized low-rank approximations for four different classes of random matrices. The analysis is based on the connection between low-rank approximations to sample covariance matrices, and using results from non-asymptotic random matrix theory. The analysis focuses on the minimal number of samples and the error in the low-rank approximation using randomized subspace iteration. The results are more generally applicable to other low-rank approximation algorithms. Of the four classes of random matrices discussed, \cref{def:randommat4} is the most general only requiring independent, isotropic columns, with bounded second moment and provides theoretical guarantees for a truly wide range of random matrices. Numerical results on a range of test matrices show that in practice the approximation errors are comparable across different distributions.

There are several avenues for future work. In ongoing work, we are developing results for random matrices using tensor random projections. We end this paper with several additional questions worth considering:
\begin{enumerate}
    \item Can we find explicit and optimal constants for the theorems presented in this work?
    \item It is known that  $\mc{O}(k \log k)$ samples are required for certain distributions; therefore, a general bound must require $\mc{O}(k \log k)$. Can this be improved to $\mc{O}(k)$ for certain distributions? We believe that it can be improved to $\mc{O}(k)$ for uniform sampling from a convex set and log-concave distributions following~\cite{adamczak2010quantitative,srivastava2013covariance}. 
    \item Can we construct random matrices that are easy to generate and store and satisfy one of the sufficient conditions in~\cref{def:randommat1,def:randommat2,def:randommat3,def:randommat4}? 
    \item Is the bounded moment condition in~\cref{def:randommat4} necessary? Can we extend our analysis to other heavy-tailed distributions?
    \item Can we improve the dependence on failure probability in~\cref{thm:indepinexp} and extraneous logarithmic factors?
\end{enumerate}

\appendix

\section{Additional background and proofs}\label{ssec:exproofs}

In Section~\ref{ssec:exproofs}, we give some additional background on Gaussian width, review some auxiliary results, and give additional proofs. In Section~\ref{ssec:otherex}, we give additional examples of distributions to illustrate the main results. Then, in Section~\ref{sec:othernum}, we present additional numerical examples, and in Section~\ref{sm:Nystrom} we present an extension of our results to the Nystr\"om approach.

\subsection{Gaussian width and complexity}\label{ssec:gwidth}
The analysis of random matrices depends on the notion of the \textit{Gaussian width}. The Gaussian width of a set $\mc{T} \subset \R^n$ is defined as 
\begin{equation}
    w(\mc{T})  :=\expect{\sup_{\B{x} \in \mc{T}} \, \inner{\B{g}}{\B{x}}}, \qquad \B{g} \sim \mathcal{N}(\B{0},\B{I}).
\end{equation}
This is a fundamental quantity that is closely related to other geometric qualities associated with $\mc{T}$ such as surface area and volume. We also define the radius of a set $\mc{T} \subset \R^n$ as $\rad(\mc{T}) := \sup_{\B{x} \in \mc{T}} \|\B{x}\|_2$.  

The Gaussian width $w(\mc{T})$ has the following properties~\cite[Proposition 7.5.2.]{vershynin2018high}: it is invariant under translations and rotations; that is if $\mat{U} \in \R^{n \times n}$ is orthogonal and $\vec{y}\in \R^n$
$ w(\mat{U}\mc{T} + \mat{y}) = w(\mc{T})$. Furthermore, it is equivalent to the diameter of a set $\mc{T}$, i.e., $ \frac{1}{\sqrt{2\pi}}\mathsf{diam}(\mc{T}) \leq w(\mc{T}) \leq \frac{\sqrt{n}}{2} \mathsf{diam}(\mc{T})$,
where the diameter $ \mathsf{diam}(\mc{T}) := 2\rad(\mc{T})$.

Related to the Gaussian width of a set $\mc{T} \subset \R^n$ is the notion of the Gaussian complexity, $\gamma(\mc{T})$, defined as \begin{equation}\label{eqn:gausscomplex}
    \gamma(\mc{T})  :=\expect{\sup_{\B{x} \in \mc{T}} \, |\inner{\B{g}}{\B{x}}|}, \qquad \B{g} \sim \mathcal{N}(\B{0},\B{I}).
\end{equation}
Obviously, $w(\mc{T}) \leq \gamma(\mc{T})$, and for any set containing the origin both are equivalent since $\gamma(\mc{T}) \leq 2 w(\mc{T})$.

Let us now state and prove an elementary inequality regarding the Gaussian width of an ellipsoidal set. This result plays an essential role in deriving dimension-independent bounds. 
\begin{lemma}\label{lem:wT} Let $\B{H}\in \R^{m\times n}$. Then $w(\B{H} \mc{S}^{n-1} ) \leq  \|\B{H}\|_F$.
\end{lemma}

\begin{proof}[Proof of \cref{lem:wT}]
Let $\B{x} = \B{Hz} \in \B{H} \mc{S}^{n-1}$. Then,
\[   {\sup_{\B{x} \in \B{H}\mc{S}^{n-1} } \, \inner{\B{g}}{\B{x}}} =  {\sup_{\B{z} \in \mc{S}^{n-1} } \, \inner{\B{g}}{\B{Hz}}} = {\sup_{\B{z} \in \mc{S}^{n-1} } \, \inner{\B{H}\t\B{g}}{\B{z}}}.\] 
Applying Cauchy-Schwartz inequality $\inner{\B{H}\t \B{g}}{\B{z}} \leq \|\B{H}\t\B{g}\|_2  \|\B{z}\|_2= \|\B{H}\t\B{g}\|_2$, for any  $\B{z} \in \mc{S}^{n-1}$. Next, applying Cauchy-Schwartz inequality for random variables yields
\[  w(\B{H} \mc{S}^{n-1} ) = \expect{\sup_{\B{z} \in \mc{S}^{n-1} } \, \inner{\B{H}\t\B{g}}{\B{z}}} \leq  \expect{\left(\|\B{H}\t\B{g}\|_2^2\right)^{1/2}} = \sqrt{\mathsf{trace}(\B{H}\t\B{H})} = \|\B{H}\|_F  \]
which completes the proof.
\end{proof}

Note that \cref{lem:wT} provides a tighter bound than the one in~\cite[Exercise 7.5.4]{vershynin2018high} since
$\|\B{H}\|_F \leq \sqrt{n}\|\B{H}\|_2$.

\subsection{Auxiliary results}

The following result~\cite[Proposition 2.6.1]{vershynin2018high} gives insight into the sum of independent sub-Gaussian random variables. 
\begin{lemma}\label{lem:subgausssum}
    Let $X_1,\dots,X_N$ be independent, mean zero, sub-Gaussian random variables. Then $\sum_{i=1}^N X_i$ is also sub-Gaussian random variable with zero mean and 
    \[ \| \sum_{i=1}^N X_i\|_{\psi_2}^2 \leq C_R\sum_{i=1}^N\| X_i\|_{\psi_2}^2. \]
\end{lemma}
This result is a form of generalization of a result for Gaussian random variables, i.e., If $X_1,\dots,X_N$ are independent Gaussian random variables, then their sum is a Gaussian random variable whose variance $\text{Var}(\sum_{i=1}^N X_i)$ is the sum of variances $\sum_{i=1}^N\text{Var}(X_i)$. 

The following results regarding sub-Gaussian random vectors will be useful in our analysis.
\begin{lemma}\label{lem:subgauss}
Let {$\mat{W} \in \R^{n\times d}$} and let $\vec{z} \in \R^n$ be a vector with independent, mean zero, sub-Gaussian entries such that $\|z_i\|_{\psi_2} \leq K $, $1 \leq i \leq n$. Then $\mat{W}^\top \vec{z}$ is sub-Gaussian, so that for every {$\vec{x} \in \R^d$}
\begin{equation}\label{eqn:subgaussl2} \|  \vec{x}^\top \mat{W}^\top \vec{z}\|_{\psi_2 }^2 \leq C_R K^2\|\vec{x}^\top \mat{W}^\top \vec{z}\|_{L_2}^2.  \end{equation}
Furthermore, if $\mat{W}$ has orthonormal columns, then $\|\mat{W}^\top \vec{z}\|_{\psi_2 }^2\leq C_R K^2$. Here $C_R$ is the constant from \cref{lem:subgausssum}.
\end{lemma}
\begin{proof}[Proof of \cref{lem:subgauss}]
For every $\vec{x} \in \R^d$, with $\vec{y} = \mat{Wx}$
$$\|  \vec{x}^\top \mat{W}^\top \vec{z}\|_{\psi_2 } = \|\vec{y}^\top\vec{z} \|_{\psi_2} = \|\sum_{i=1}^n y_i z_i \|_{\psi_2}. $$
Therefore, the marginal is a weighted sum of sub-Gaussian random variables and, therefore, sub-Gaussian. Using~\Cref{lem:subgausssum} 
\begin{equation}
    \label{eqn:hoeff}
    \|\sum_{i=1}^n y_i z_i \|_{\psi_2}^2 \leq C_R \sum_{i=1}^n y_i^2 \|z_i\|_{\psi_2}^2 \leq C_R \|\vec{y}\|_2^2 \max_{1 \leq i \leq n} \|z_i\|_{\psi_2}^2, 
\end{equation}
where $C_R$ is the constant in \cref{lem:subgausssum}. 
Note that 
\[  \|\vec{x}^\top \mat{W}^\top \vec{z}\|_{L_2}^2 = \left(\expect{\vec{x}^\top \mat{W}^\top \vec{z}}^2 \right) = \left(\vec{x}^\top \mat{W}^\top \mat{Wx}   \right) = \|\vec{y}\|_2^2,\]
which when combined with~\eqref{eqn:hoeff} gives~\eqref{eqn:subgaussl2}. 

For the second part,  $\vec{y} \in \R^n$ has unit norm since $\mat{W}$ has orthonormal columns. Therefore, $\mat{W}^\top \vec{z}$ is sub-Gaussian with norm $\|\mat{W}^\top \vec{z}\|_{\psi_2 }^2\leq C_RK^2  \text{sup}_{\vec{y} \in \mc{S}^{{d-1}}} \|\vec{y}\|_2^2  = C_RK^{{2}}$.
\end{proof}

\begin{lemma}\label{lem:subgauss2} Let $\vec{x} \in \R^n$ be a sub-Gaussian random vector with $\|\vec{x}\|_{\psi_2} \leq K$ and let {$\mat{W} \in \R^{n\times d}$} be a matrix with orthonormal columns. Then $\|\mat{W}\t\vec{x}\|_{\psi_2}^2 \leq C_R K^2$ where $C_R$ is the constant in \cref{lem:subgausssum}.
\end{lemma}

\begin{proof}[Proof of \cref{lem:subgauss2}]
    The proof follows from 
\[\begin{aligned}\| \mat{W}\t\vec{x} \|_{\psi_2}^2 = & \> \sup_{\vec{z}\in \mc{S}^{d-1}} \| \inner{\mat{W}\t\vec{x}}{\vec{z}} \|_{\psi_2}^2 = \sup_{\vec{z}\in \mc{S}^{d-1}} \| \inner{\vec{x}}{\vec{Wz}} \|_{\psi_2}^2\\
\leq  & \>  \sup_{\vec{y}\in \mc{S}^{n-1}} \| \inner{\vec{x}}{\vec{y}} \|_{\psi_2}^2 \leq C_R K^2. \end{aligned} \] 
In the last step, we have used \cref{lem:subgausssum} and the fact that $\|\vec{y}\|_2 = 1$.
\end{proof}

This result says that if $\vec{z} \in \R^n$ has independent sub-Gaussian entries with norm $K$, then $\mat{W}^\top\vec{z}$ (with $\mat{W}$ having orthonormal columns) is also sub-Gaussian with norm bounded by $K$ up to a constant factor. The next result extends it to sub-Gaussian random vectors.

\subsection{Random matrix bounds} \label{ssec:rmbounds}

\begin{theorem}[Talagrand's comparison inequality]\label{thm:Talagrand}
    Let $X_{\vec{t}}$ for $\vec{t}\in \mc{T} \subset \R^n$ be a random process. Assume that for all $\vec{x},\vec{y}\in \mc{T}$ we have $ \|X_{\vec{x}} - X_{\vec{y}}\|_{\psi_2}\leq K_T\|\vec{x} - \vec{y}\|_2$. 
    Then the following hold:
    \begin{enumerate}
        \item Expectation bound: $\expect{\sup_{\vec{t}\in \mc{T}} X_{\vec{t}}} \leq C_{TE} K_Tw(\mc{T})$.
        \item {Tail bound}: For every $u \geq 0$,  with probability, at least $1-2\exp(-u^2)$$$\sup_{\vec{t}\in \mc{T}}| X_{\vec{t}}| \leq C_{TT}K_T(w(\mc{T}) + u\cdot \rad(\mc{T})).$$
    \end{enumerate}
    Here $C_{TE}$ and $C_{TT}$ are absolute constants.
\end{theorem}
\begin{proof}
{See \cite[Corollary 8.6.3, Exercise 8.6.5]{vershynin2018high}.}
\end{proof}
We briefly remark that explicit (but suboptimal constants) are available in~\cite[Remark 3.3(iv)]{dirksen2015tail}.


\begin{proof}[Proof of \Cref{thm:subgaussnorm}]
First, w.l.o.g., let us consider $\|\B{M}\|_2 =  \|\B{N}\|_2  = 1$. Thus $\sr(\B{M}) = \|\B{M}\|_F^2$ and $\sr(\B{N}) = \|\B{N}\|_F^2$.  The main idea is to use Talagrand's comparison inequality, see \Cref{thm:Talagrand}, which was used to establish the sub-Gaussian Chevet inequality~\cite[Theorem 8.7.1]{vershynin2018high}.

\paragraph{Step 1: Setup the stochastic process} Let us define $\mc{U} := {\B{N}} \mc{S}^{N-1}$ and $\mc{V} := \B{M}\t \mc{S}^{M-1}$ (ellipsoids of radius $1$) and consider the random process 
\[ X_{\B{uv}} := \inner{\B{S} \B{u}}{\B{v}}, \qquad \B{u} \in \mc{U}, \B{v}\in \mc{V}. \]
It is easy to check that $X_{\B{uv}}$ has zero mean and 
$${\sup_{\B{u}\in \mc{U}, \B{v}\in \mc{V}} X_{\B{uv}}}  = \sup_{ \B{x} \in \mc{S}^{n-1}, \B{y} \in \mc{S}^{m-1} }\inner{\B{MSN} \B{x}}{\B{y}}  = {\|\B{MSN}\|_2}. $$ 

\paragraph{Step 2: Verify sub-Gaussian increments}
Here, we develop a bound for the sub-Gaussian increment $\|X_{\B{uv}}- X_{\B{wz}}\|_{\psi_2}$ as in the proof of~\cite[Theorem 8.7.1]{vershynin2018high}. Consider $X_{\B{uv}}$ defined in \textit{Step 1}. Then, for $\B{u}, \B{w} \in \mc{U}$ and $\B{v}, \B{z} \in \mc{V}$
\[ 
\|X_{\B{uv}}- X_{\B{wz}}\|_{\psi_2} =
\|\inner{\B{S} \B{u}}{\B{v}} - \inner{\B{S} \B{w}}{\B{z}}\|_{\psi_2} =
\| \sum_{i,j} s_{ij} ({u}_j{v}_i - {w}_j {z}_i)  \|_{\psi_2}.
\]
Since the entries of matrix $\B{S}$ are independent sub-Gaussian random variables with mean zero and sub-Gaussian norm $\|s_{ij}\|_{\psi_2} \leq K$, by the concentration inequality \Cref{lem:subgausssum}
\[ \begin{aligned}\|\sum_{i=1}^m\sum_{j=1}^n s_{ij} ({u}_j{v}_i -  {w}_j  {z}_i)  \|_{\psi_2}  \leq & \>   \sqrt{C_R}  \left( \sum_{i=1}^m\sum_{j=1}^n\|  s_{ij} ( {u}_j {v}_i -  {w}_j  {z}_i)  \|_{\psi_2}^2\right)^{1/2} \\
\leq & \>  \sqrt{C_R}K \| \B {u}\B {v}\t- \B {w}\B {z}\t\|_F.\end{aligned}\]
Following the proof of~\cite[Theorem 8.7.1]{vershynin2018high} we get
$$
\begin{aligned}
\| \B {u}\B {v}\t- \B {w}\B {z}\t\|_F \leq & \> \sqrt{2} \left( \| \B {u} -\B {w}\|_2^2  + \| \B {v}-\B {z}\|_2^2\right)^{1/2}.
\end{aligned}$$
Combining intermediate results  gives the following bound
\[ \|X_{\B{uv}}- X_{\B{wz}}\|_{\psi_2} \leq \sqrt{2C_R}K\left( \| \B{u} -\B{w}\|_2^2  + \| \B{v}-\B{z}\|_2^2\right)^{1/2}. \] 
\paragraph{Step 3: Expectation bound} 
As we have verified the requirements of~\Cref{thm:Talagrand} 
\begin{equation}
    \label{eqn:talagexpect}\expect{\sup_{(\B{u}, \B{v}) \in \mc{U} \times \mc{V}} X_{\B{uv}}}  \leq \sqrt{2C_R} C_{TE}K w(\mc{U}\times \mc{V}).
\end{equation} 
A quick calculation shows that for $\B{g}, \B{h} \sim \mathcal{N}(\B{0},\B{I})$
\begin{equation}
\label{eq:widthUV}
\begin{aligned} w(\mc{U} \times \mc{V}) = & \> \expect{ \sup_{(\B{u}, \B{v}) \in \mc{U} \times \mc{V}}  \inner{\B{g}}{\B{u}} +  \inner{\B{h}}{\B{v}}} =  \> \expect{ \sup_{\B{u} \in \mc{U}} \inner{\B{g}}{\B{u}}} + \expect{ \sup_{\B{v} \in \mc{V}} \inner{\B{h}}{\B{v}}} \\
\leq & \>   \|{\B{M}}\|_F + \|{\B{N}}\|_F.\end{aligned}
\end{equation}
Here, we have used the linearity of expectation values and Lemma~\ref{lem:wT}. Combining bound~\eqref{eq:widthUV} with~\eqref{eqn:talagexpect} proves~\eqref{eqn:subgaussexpect} with $C_{SNE} := \sqrt{2C_R}C_{TE}$. 
\paragraph{Step 4: Tail Bound}
In the following, we use the tail-bound version of Talagrand's comparison inequality~\Cref{thm:Talagrand}. For every $u \geq 0 $
\begin{equation}\label{eqn:talagtail}  \sup_{(\B{u},\B{v}) \in \mc{U} \times \mc{V}}X_{\B{uv}}  \leq \sup_{(\B{u},\B{v}) \in \mc{U} \times \mc{V}}|X_{\B{uv}} |\leq  \sqrt{2C_R}C_{TT}K \Big( w(\mc{U} \times \mc{V}) + u\ \rad(\mc{U} \times \mc{V})\Big),
\end{equation}
with probability at least $1- 2\exp(-u^2)$.
Next, inserting~\eqref{eq:widthUV} 
and the inequality$$\rad(\mc{U}\times \mc{V}) \leq \max\{\|\B{M}\|_2,\|\B{N}\|_2 \} = 1,$$ into~\eqref{eqn:talagtail} yields~\eqref{eqn:subgausstail}, and completes the proof with $C_{SNT} := \sqrt{2C_R}C_{TT}$.
\end{proof}

\begin{theorem}[Matrix deviation inequality]\label{thm:matdev}
    Let $\mat{S} \in \R^{m\times n}$ whose rows $\vec{e}_i^\top\mat{S}$ are independent, isotropic, and sub-Gaussian random vectors in $\R^n$. Then for any subset $\mc{T} \subset \R^n$, we have 
    \[\expect{ \sup_{\vec{x}\in \mc{T}} | \|\mat{Sx} \|_2 - \expect{\|\mat{Sx}\|_2}|}  \leq C_{MDE} K^2 \gamma(\mc{T}).\]
    Here $\gamma(\mc{T})$ is the Gaussian complexity and $K := \max_{1 \leq i \leq  m}\|\vec{e}_i^\top\mat{S}\|_{\psi_2} .$ Furthermore, for any $u \geq 0$, the inequality 
    \[\sup_{\vec{x}\in \mc{T}} | \|\mat{Sx} \|_2 - \sqrt{\|\vec{x}\|_2}|\leq C_{MDT}K^2 [w(\mc{T}) +u\cdot \rad(\mc{T})],   \]
    holds with probability at least $1-\exp(-u^2)$. Here $C_{MDE}$ and $C_{MDT}$ are absolute constants.
\end{theorem}
\begin{proof}
    See~\cite{vershynin2010introduction}, Theorem 9.1.1. and Exercise 9.1.8.
\end{proof}

\begin{proof}[Proof of \Cref{thm:subgaussnorm2}]
Note that $\B{S}\t$ has independent sub-Gaussian isotropic rows and consider $\mc{T} = \B{M}\t \mc{S}^{M-1}$. The expectation bound  in \Cref{thm:matdev} and the triangle inequality implies
\begin{equation}
\label{eq:Thm15Exp}
\begin{aligned}  \expect{\sup_{\B{x} \in \mc{T}} \|\B{S}\t\B{x}\|_2 } \leq &\>   \expect{\sup_{\B{x} \in \mc{T}}| \|\B{S}\t\B{x}\|_2 - \sqrt{n}\|\B{x}\|_2 |} + \sqrt{n}\sup_{\B{x}\in \mc{T}}\|\B{x}\|_2 \\\leq  & \>C_{MDE}K^2 \gamma(\mc{T}) + \sqrt{n} \ \rad(\mc{T}).\end{aligned} 
\end{equation} 
By an argument similar to the proof of~\Cref{lem:wT}, we have $\gamma(\mc{T}) \leq \|\B{M}\|_F$ and $\rad(\mc{T}) \leq \|\B{M}\|_2$. Plugging these into \eqref{eq:Thm15Exp} gives the expectation bound~\eqref{eqn:subgaussexpect2}.

Furthermore, applying the tail bound of \Cref{thm:matdev}, with probability at least $1 - 2\exp(-u^2)$
\[  \sup_{\B{x} \in \mc{T}}| \|\B{S}\t\B{x}\|_2 - \sqrt{n}\|\B{x}\|_2 | \leq C_{MDT}K^2 [w(\mc{T}) + u\cdot\rad(\mc{T})]. \]
 Therefore, once again using triangle inequality and~\Cref{lem:wT}, we obtain
\[ \begin{aligned} \sup_{\B{x} \in \mc{T}}\|\B{S}\t\B{x}\|_2 \leq & \>  \sqrt{n} \sup_{\B{x} \in T} \|\B{x}\|_2 +  C_{MDT}K^2 [{\|\B{M}\|_F} + u \|\B{M}\|_2] \\ 
\leq & \>\sqrt{n} \|\B{M}\|_2 +  C_{MDT}K^2 [{\|\B{M}\|_F} + u \|\B{M}\|_2],
 \end{aligned} \] 
i.e., the tail bound~\eqref{eqn:subgausstail2}. 
\end{proof}

\section{Other examples}\label{ssec:otherex}
In this section, we give additional examples of distributions to illustrate the bounds presented in Section~\ref{sec:main}. 
\subsection{Independent sub-Gaussian entries}\label{ssec:supp_mat1}
We also consider another distribution which we call \textit{sparse sub-Gaussian distribution} (related to~\cite[Example 2.9]{oymak2018universality}). This distribution may be computationally advantageous in some cases since the matrix-vector products with matrix $\B\Omega \in \R^{n \times \ell}$ may be cheaper because of many zero entries. Let $\alpha \in (0,1]$ be a user-defined sparsity parameter and let $Y$ be a Bernoulli random variable that takes the values $1$ with probability $\alpha$ and $0$ with probability $1-\alpha$. Let $Z$ be a sub-Gaussian random variable with zero mean, unit variance, and sub-Gaussian norm {bounded by} $K_Z$. Then it is easy to verify that $X = \alpha^{-1/2}YZ$ is a sub-Gaussian random variable with zero mean, unit variance, and with sub-Gaussian norm {bounded by} $CK_Z/\sqrt{\alpha}$, where $C$ is some absolute constant. Let $\mat\Omega \in \R^{n \times \ell}$ be a random matrix with independent copies of $X$; then {by~\Cref{thm:subgauss}}, the number of required samples $\ell$ is $\mc{O}(K_Z^4 k/ \alpha^2)$. A special case is the random sparse sign matrix, for which $Z$ is a Rademacher random variable (whose sub-Gaussian norm is an absolute constant), and the number of required samples for this distribution $\ell$ is $\mc{O}(k/ \alpha^2)$. This requirement is similar to the one for the sparse Rademacher random matrix with $\alpha = 1/s$.

\subsection{Independent sub-Gaussian columns}\label{ssec:supp_mat2}

We consider a new random matrix model called {\em sparse sign matrices} Consider a random matrix $\mat\Omega \in \R^{n\times \ell} = \sqrt{\frac{n}{N}} \bmat{\vec{s}_1 & \dots & \vec{s}_\ell}$ where each column has entries $\pm 1$ (with equal probability) situated in $N$ nonzero locations, chosen uniformly at random from $1$ to $n$. This is a slightly different model than~\cite{cohen2016nearly}  in which the author uses independent rows rather than columns. It can be verified using {\Cref{lem:subgausssum}}, applied conditionally on the sparsity pattern, that the sub-Gaussian norm of each column {is bounded by} $K_C = C\sqrt{C_Rn/N}$, where $C$ is an absolute constant.  Therefore, by \cref{thm:subgauss2}, the number of samples required is $\mc{O}(k (n/N)^2)$.

\subsection{Independent bounded columns}\label{ssec:supp_mat3}
We now consider the coordinate distribution. In this case, the columns of $\mat\Omega \in \R^{n \times \ell}$ are drawn uniformly from the set $\{\sqrt{n}\vec{e}_i\}_{i=1}^n$. As mentioned earlier, although the random vectors are sub-Gaussian, the sub-Gaussian norm is bounded by $\sqrt{n}$, making the bounds in \cref{thm:subgauss2} impractical. We now consider the consequence of \cref{thm:genindep} for the coordinate distribution. Given a matrix $\mat{W} \in \R^{n\times d}$ with orthonormal columns,  we define the coherence of $\mat{W}$ as the largest squared row norm (largest leverage score), i.e.,
\begin{equation}\label{eqn:coherence} \mu(\mat{W}) := \max_{1 \leq j \leq n} \|\mat{W}^\top\vec{e}_j\|_2^2. \end{equation}
The coherence takes values in the range {$[\frac{d}{n},1]$}. The two extreme cases $\frac{d}{n}$ and $1$ can be achieved 
when all the columns of $\mat W$ have equal norm, e.g., columns from the Hadamard matrix, and the identity matrices, respectively. Then the boundedness assumption in \cref{thm:genindep} takes the form
\[  \|\mat{V}_k^\top \B\Omega\vec{e}_j\|_2 \leq \sqrt{n \mu(\mat{V}_k)} \quad \mbox{ and } \quad \|\mat{V}_\perp^\top \B\Omega\vec{e}_j\|_2  \leq \sqrt{n \mu(\mat{V}_\perp)}      ,  \]
for all $1 \leq j \leq \ell$. The number of required samples is, therefore, $\ell = \frac{ {2n \mu(\mat{V}_k)}}{\varepsilon^2} \log(2k/\delta). $
This result means that many samples are required if the matrix $\mat{V}_k$ has high coherence. At the very minimum, it requires $\mc{O}(k \log k)$ samples, but a matrix with high coherence could require the number of samples to depend on the dimension $n$. Similar observations can be made in the case of the random vectors drawn uniformly from the columns of an orthogonal matrix.

\subsection{Independent columns with bounded second moments}\label{ssec:supp_mat4}
We first explain why this model is the most general of those we considered. Note that any random matrix that satisfies \cref{def:randommat3} automatically satisfies \cref{def:randommat4} with $K_M = K_k^2 $. Similarly, any random matrix that satisfies \cref{def:randommat2} (and, therefore, \cref{def:randommat1})  also satisfies \cref{def:randommat4} with $K_M = C_RK_C^2 \log(2\ell)$. To see this, we use the approach described in~\cite[Chapter 2.5]{boucheron2013concentration}.  Let $\mat{W}\in \R^{n \times d}$ have orthonormal columns. From \cref{lem:subgauss2}, the columns of $\mat{W}\t\mat{\Omega}$ have the sub-Gaussian norm bounded by $C_RK_C$. This means that for each $1 \leq j \leq \ell$, $\expect{\exp(\|\mat{W}\t\mat\Omega\vec{e}_j\|_2^2/ (C_RK_C^2))} \leq 2$. By Jensen's inequality, with $\lambda = 1/(C_R K_C^2)$, 
\begin{equation}\label{eqn:maximalinequality}\begin{aligned} \exp\left( \lambda \expect{\max_{1 \leq j \leq \ell} \|\mat{W}^\top \mat\Omega\vec{e}_j\|_2^2}\right) \leq & \> \expect{\exp\left(\lambda \max_{1\leq j \leq \ell} \|\mat{W}^\top \mat\Omega\vec{e}_j\|_2^2\right)} \\
= & \> \expect{ \max_{1 \leq j \leq \ell }e^{\lambda \|\mat{W}^\top \mat\Omega\vec{e}_j\|_2^2}} \\
\leq & \sum_{j=1}^\ell \expect{e^{\lambda \|\mat{W}^\top \mat\Omega\vec{e}_j\|_2^2}} \leq 2\ell. \end{aligned} \end{equation}
Taking logarithms, we get $\expect{\max_{1 \leq j \leq \ell} \|\mat{W}^\top \mat\Omega\vec{e}_j\|_2^2} \leq C_R K_C^2 \log(2\ell) $. By applying this result with $\mat{W} = \mat{V}_k$, we get $K_M := C_R K_C^2 \log(2\ell)$. Therefore, \cref{def:randommat4} is the most general case that we consider in this paper.

We now provide a proof for the claim in \cref{cor:alpha}.

\begin{proof}[Proof of \cref{cor:alpha}]

First, we provide a bound for $K_M$ in \cref{def:randommat4}. By the Hanson-Wright inequality~\cite[Proposition 1.1]{Goetze_2021}
\[ \prob{ |\;\|\mat{V}_k\t\mat\Omega\vec{e}_j\|_2^2 - k| \geq t} \leq \exp\left(-\min\left\{\frac{t^2}{M^4k}, \frac{t^{\alpha/2}}{M^\alpha}\right\}\right).\] 
Define the random variable $Z =  \max\limits_{1 \leq j \leq \ell} |\; \|\mat{V}_k\t\mat\Omega\vec{e}_j\|_2^2 - k|.$ By the triangle inequality, $K_M \leq \expect{Z} + k$, so it is sufficient to bound $\expect{Z}$. By the union bound
\[\prob{Z \geq t} \leq \ell\exp\left(-\min\left\{\frac{t^2}{M^4k}, \frac{t^{\alpha/2}}{M^\alpha}\right\}\right) .\]
Then using $\expect{Z} = \int_0^\infty {\prob{Z \geq t}dt}$ since $Z$ is nonnegative, we have for some $\mu > 0$
\begin{equation}\label{eqn:alphainter} \begin{aligned}\expect{Z} \leq & \> \int_0^\mu dt + \int_\mu^\infty \ell\exp\left(-\min\left\{\frac{t^2}{M^4k}, \frac{t^{\alpha/2}}{M^\alpha}\right\}\right)dt \\
\leq & \>  \mu +  {n}\max\left\{\int_\mu^\infty \exp\left(-\frac{t^2}{M^4k}\right) dt,  \int_\mu^\infty  \exp\left(- \frac{t^{\alpha/2}}{M^\alpha}\right) \right\} dt. 
\end{aligned} \end{equation}
We tackle each term separately. First, we use the inequality  $\int_\mu^\infty e^{-t^2/(2\beta)}dt \leq \sqrt{\beta} e^{-t^2/(2\beta)}dt$ (\cite[Proof of Corollary 7.3.2]{tropp2015introduction}) to obtain 
\[\int_\mu^\infty \exp\left(-\frac{t^2}{M^4k}\right) dt \leq M^2\sqrt\frac{k}{2} \exp\left(-\frac{\mu^2}{M^4k}\right).  \] 
Now set $n\sqrt\frac{k}{2} \exp\left(-\frac{\mu^2}{M^4k}\right) = 1$
and solve for $\mu$ to get $\mu =  M^2 \sqrt{k\log(n \sqrt{k/2})} .$ For the other term, we use~\cite[Proof of Lemma A.2]{Goetze_2021}
\[ \int_\mu^\infty  \exp\left(- \frac{t^{\alpha/2}}{M^\alpha}\right)  \leq  \left( \frac{2}{\alpha e}\right)^{2/\alpha} \int_\mu^\infty (t/M^2)^{-1}dt
= M^2\left( \frac{2}{\alpha e}\right)^{2/\alpha} \frac{1}{2\mu^2}. \] 
Since $\mu \geq 1$ by assumptions on $n$ and $M$, it is clear that we can bound this term by $M^2 C_\alpha$ where $C_\alpha := \frac{1}{2}\left( \frac{2}{\alpha e}\right)^{2/\alpha} $. Now combining with~\eqref{eqn:alphainter},  $\expect{Z} \leq \mu + M^2 \left(1 + C_\alpha\right)$. Therefore, $K_M := k + \mu + M^2 \left(1 + C_\alpha\right)$ where $\mu =  M^2\sqrt{k\log(n\sqrt{k/2})}$. Using \cref{thm:indepinexp}, the number of samples is $\ell = \mc{O}((k + M^2\sqrt{k\log(n)} + M^2(2/\alpha)^{2/\alpha})\log k)$.

\end{proof}

We now discuss two other classes of random matrices that satisfy \cref{def:randommat4}.
\paragraph{Uniform on an isotropic convex set} The random vectors from this distribution are uniformly distributed in a bounded convex set $\mc{K}$ with a non-empty interior. Examples of such sets include the cube $\mc{K} =  [-\sqrt{3},\sqrt{3}]^n$ (already discussed in~\cref{ssec:subgauss1}) and the unit $\ell_p$ ball in $\R^n$ for $p \geq 1$. A convex set is said to be in {\em isotropic position} if it is symmetric, with the center of mass at the origin, and satisfying
\[ \frac{1}{\text{vol}(\mc{K})} \int_{\mc{K}} \langle \vec{x}, \vec{y}\rangle^2 d\vec{x} = \|\vec{y}\|_2^2,\]
for all $\vec{y} \in \R^n$. If a random vector $\vec{x}$ that is uniformly drawn from $\mc{K}$ (denoted $\vec{x} \in \mc{U}(\mc{K})$) is isotropic, $\mc{K}$ is said to be in isotropic position. Furthermore, if $\mat{W}\in \R^{{n \times d}}$ has orthonormal columns, it is easy to check that the set $\mat{W}\t\mc{K}$ is in isotropic position in $\R^d$.

For certain isotropic convex sets (called $\psi_2$ bodies), a random vector uniform in the convex set $\mc{K}$ has a sub-Gaussian norm that is $\mathcal{O}(1)$. However, not all random vectors uniformly distributed on a convex set are sub-Gaussian (see~\cite[Section 3.4.4]{vershynin2018high}).

A random matrix $\mat\Omega$ whose columns are drawn uniformly from a convex set in isotropic position satisfies \cref{def:randommat4} as we show now. The first two properties follow readily, whereas, for the third, we note that~\cite[Proof of Corollary 4.1]{rudelson1999random}  for any $\vec{x} \sim \mc{U}(\mat{W}\t\mc{K})$, 
\[ \expect{\exp(\|\vec{x}\|_2^2/ Cd)} \leq 2, \]
where $C$ is an absolute constant. Therefore, using a similar argument as in~\eqref{eqn:maximalinequality}, we have  $\expect{\max_{1 \leq j \leq \ell} \|\mat{V}_k\t \mat\Omega\vec{e}_j\|_2^2} \leq Ck \log (2\ell) \leq Ck \log(2n)$. Therefore, the random matrix satisfies~\Cref{def:randommat4} with {$K_M = Ck \log(2n)$}, and so the number of required samples $\ell$ using \cref{thm:indepinexp} is $\mc{O}(k\log k\log n)$.

\paragraph{Log-concave distributions}
A random vector $\vec{x}$ is said to be log-concave if the probability distribution function is proportional to $\exp(-V(\vec{x}))$, where $\log V(\vec{x})$ is a convex function on $\R^n$. This class includes several cases we have already discussed including normal, exponential, Laplace, and uniform on a convex set. However, it also includes other distributions such as Gamma and extreme value distributions (Gumbel, Frechet, etc.). We derive the minimal number of samples for this class of distributions. Using Paouris concentration inequality~\cite[Section 1.4]{srivastava2013covariance}
\[ \prob{ \|\mat{V}_k\t\mat\Omega\vec{e}_j\|_2^2 \geq t} \leq \exp(-ct), \qquad t \geq Ck,  \] 
for absolute constants $c,C > 0$. Therefore, using the union bound 
\[ \prob{ \max_{1 \leq j \leq \ell} \|\mat{V}_k\t\mat\Omega\vec{e}_j\|_2^2 \geq t} \leq n\exp(-ct), \qquad t \geq Ck.\]
Then, for $\mu \geq Ck$
\[ \expect{\max_{1 \leq j \leq \ell} \|\mat{V}_k\t\mat\Omega\vec{e}_j\|_2^2} \leq \int_0^\mu dt + \int_{\mu}^\infty n\exp(-ct)dt = \mu + \frac{n}{c}\exp(-c\mu).  \] 
Choose $\mu = \max\{Ck, c^{-1} \log(n/c)\}$, so that $K_M := \max\{Ck, c^{-1} \log(n/c)\} + 1$. By~\cref{thm:indepinexp}, the number of samples $\ell$ is $\mc{O}( (k + \log n) \log k). $

\paragraph{Other distributions} As mentioned previously, our analysis is applicable to any distribution with independent isotropic columns and bounded second moments. However, obtaining bounds for $K_M$ may not be straightforward due to the presence of the maximum inside the expectation. To see the challenge, consider the naive inequality 
\[ \expect{\max_{1 \leq i \leq n} \|\mat{V}_k\t\mat\Omega\vec{e}_j\|_2^2 } \leq \sum_{j=1}^\ell \expect{\|\mat{V}_k\t\mat\Omega\vec{e}_j\|_2^2} \leq nk =: K_M, \]
that gives the minimal number of samples as $\ell = \mc{O}(nk\log k)$, which is extremely pessimistic. A bound for $K_M$ can be obtained using a technique known as a maximal inequality (see, e.g.,~\cite[Section 2.5]{boucheron2013concentration}). In practice, we need some stronger requirements on the random matrices (e.g., $\|\mat{V}_k\t\mat\Omega\vec{e}_j\|_2^2$ should satisfy the condition (SR) in~\cite{srivastava2013covariance}) to obtain meaningful bounds for $K_M$. An investigation of this is left to future work.

We can also extend the results to distributions with more than four moments (so-called $4 + \varepsilon$ moments) using the approach in~\cite[Section 1.5]{srivastava2013covariance} to obtain a bound for $K_M$.

\section{Other numerical experiments}\label{sec:othernum}
We give additional numerical experiments to illustrate other choices of random matrices. The setup of these numerical experiments is similar to those in \cref{ssec:exp1}. 
\paragraph{Sparse Rademacher}In this experiment, we consider the Sparse Rademacher distribution with the sparsity parameter $s$. We consider the relative error with increasing sample size for different sparsity parameters $s=1,\sqrt{3},10,50$. The choice of $s=\sqrt{3}$ is motivated by the paper~\cite{achlioptas2003database} and sometimes bears the name Achlioptas distribution.  By \cref{thm:subgauss}, the number of samples is $\mc{O}(s^2k)$, so larger values of $s$ should result in larger errors for the sample size. In \cref{fig:sparserad}, as before, we plot in \cref{fig:sparserad} the mean of the relative error and one standard deviation of the error over $100$ samples. For the \texttt{FastDecay} matrix, all four distributions have similar behavior but higher values of $s$ result in a larger standard deviation around index $15$. The differences between the distributions are more pronounced in the \texttt{ControlledGap} matrix where higher values of $s$ result in both larger means and higher variances. 
\begin{figure}[!ht]
    \centering
    \includegraphics[scale=0.33]{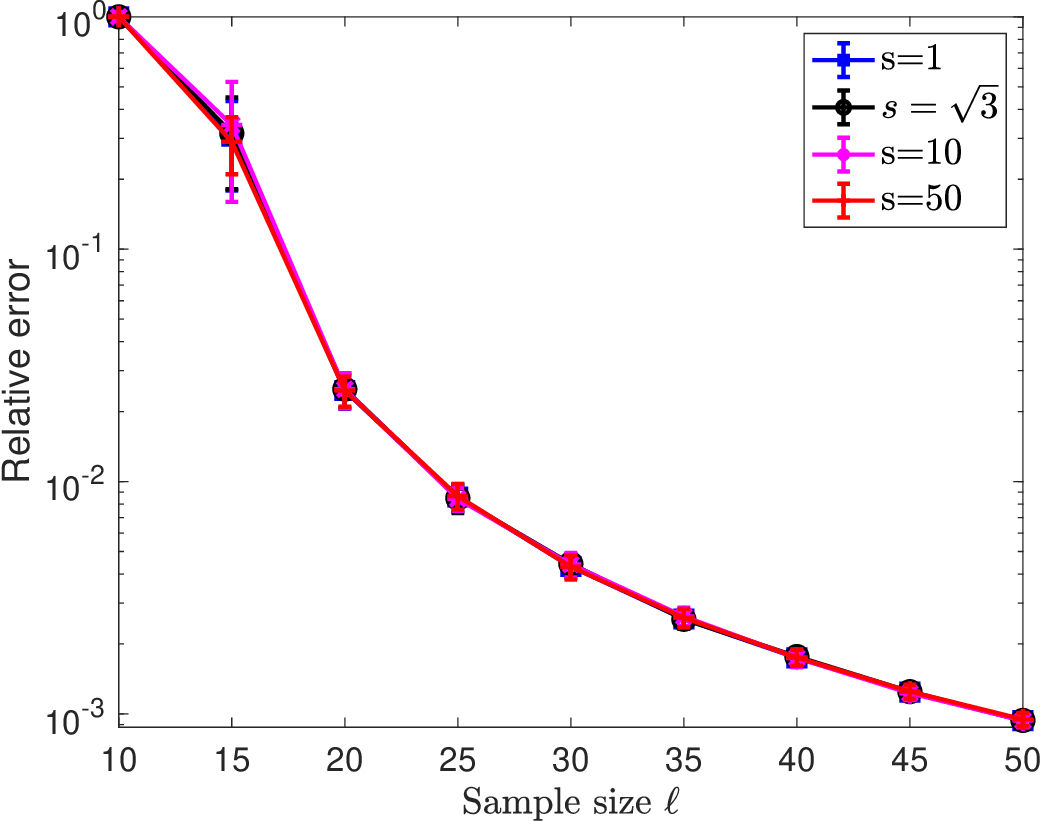}
    \includegraphics[scale=0.33]{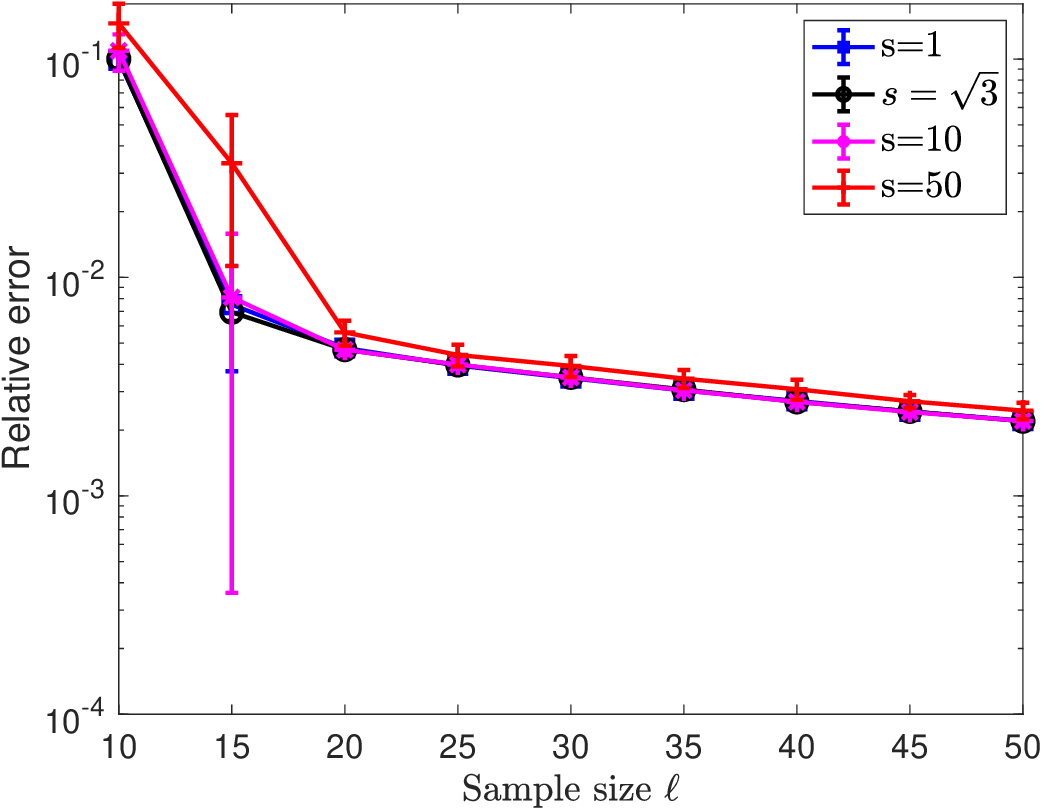}
    \caption{Relative error vs. sample size $\ell$ for the \texttt{FastDecay} (left) and \texttt{ControlledGap} (right). We use random matrices constructed using the sparse Rademacher distribution.}
    \label{fig:sparserad}
\end{figure}

\paragraph{Sparse-sign}In this experiment, we consider the sparse sign distribution with the sparsity parameter $N$. We consider the relative error with increasing sample size for different sparsity parameters $N=1,5,10,50$. By \cref{thm:subgauss}, the number of samples is $\mc{O}(k (n/N)^2)$, so larger values of $N$ should result in lower errors for the same sample size. In \cref{fig:sparserad}, as before, we plot the mean of the relative error and one standard deviation of the error over $100$ samples. For the \texttt{FastDecay} matrix, all four distributions have similar behavior. The differences between the distributions are more pronounced in the \texttt{ControlledGap} matrix where smaller $N$ results in both larger means and higher variances. 
\begin{figure}[!ht]
    \centering
    \includegraphics[scale=0.33]{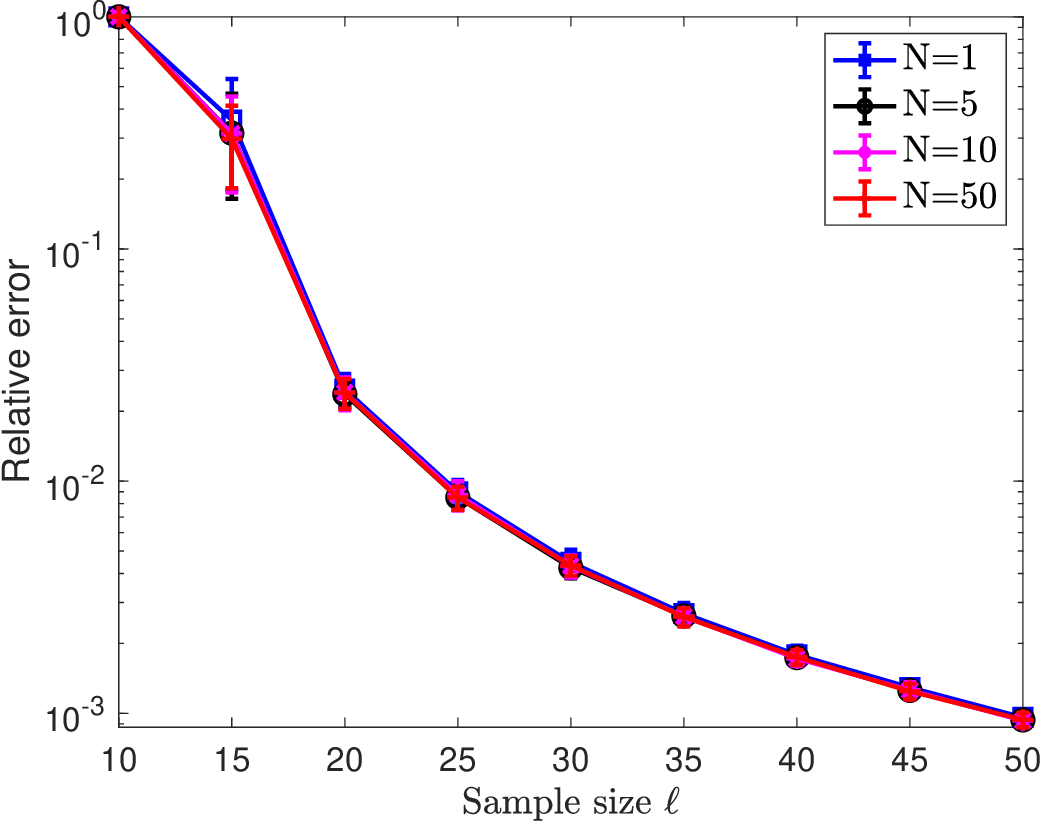}
    \includegraphics[scale=0.33]{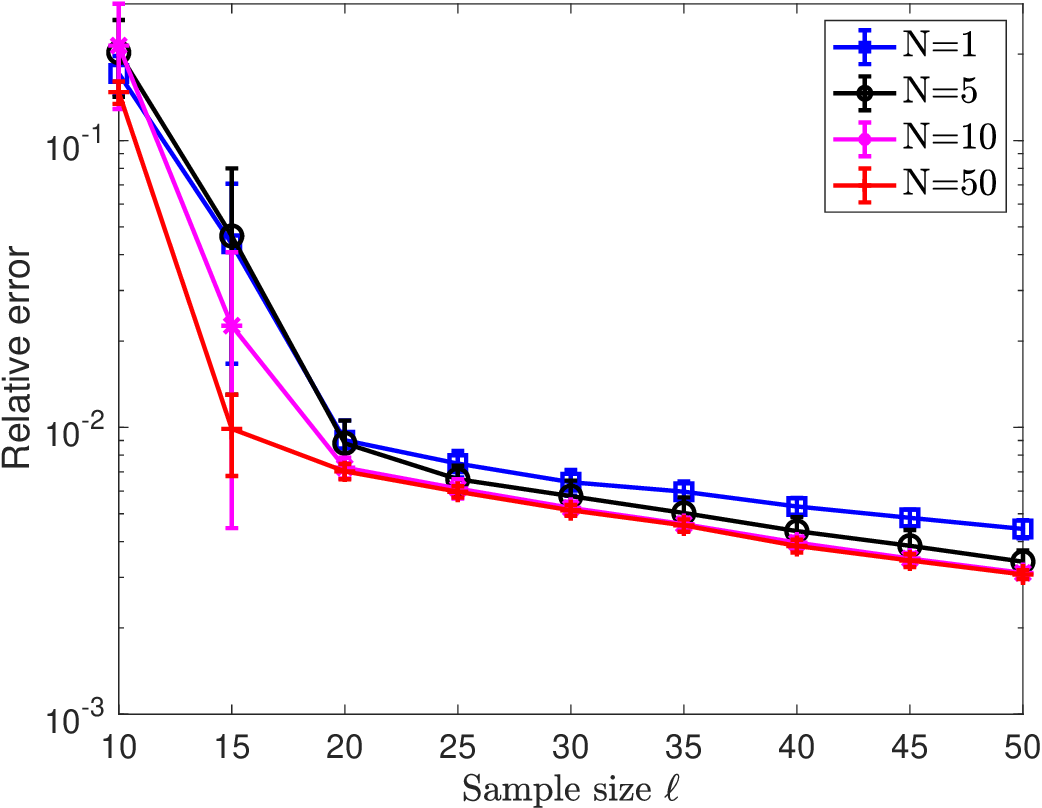}
    \caption{Relative error vs. sample size $\ell$ for the \texttt{FastDecay} (left) and \texttt{ControlledGap} (right). We use random matrices constructed using the sparse sign distribution.}
    \label{fig:sparsesign}
\end{figure}

\paragraph{Other heavy-tailed distribution}
We consider the following distributions: (1) Cauchy distribution, (2) Student's t-distribution $t=10$, (3) centered Gamma distribution $a=3, b= 5$, (4) Stable ($\alpha =1, \beta = \delta = 0, \gamma = 1$). Except for the Gamma distribution, these are considered heavy-tailed distributions. All other settings are as in Experiment 1 (\cref{ssec:exp1}).

\begin{figure}[!ht]
    \centering
    \includegraphics[scale=0.33]{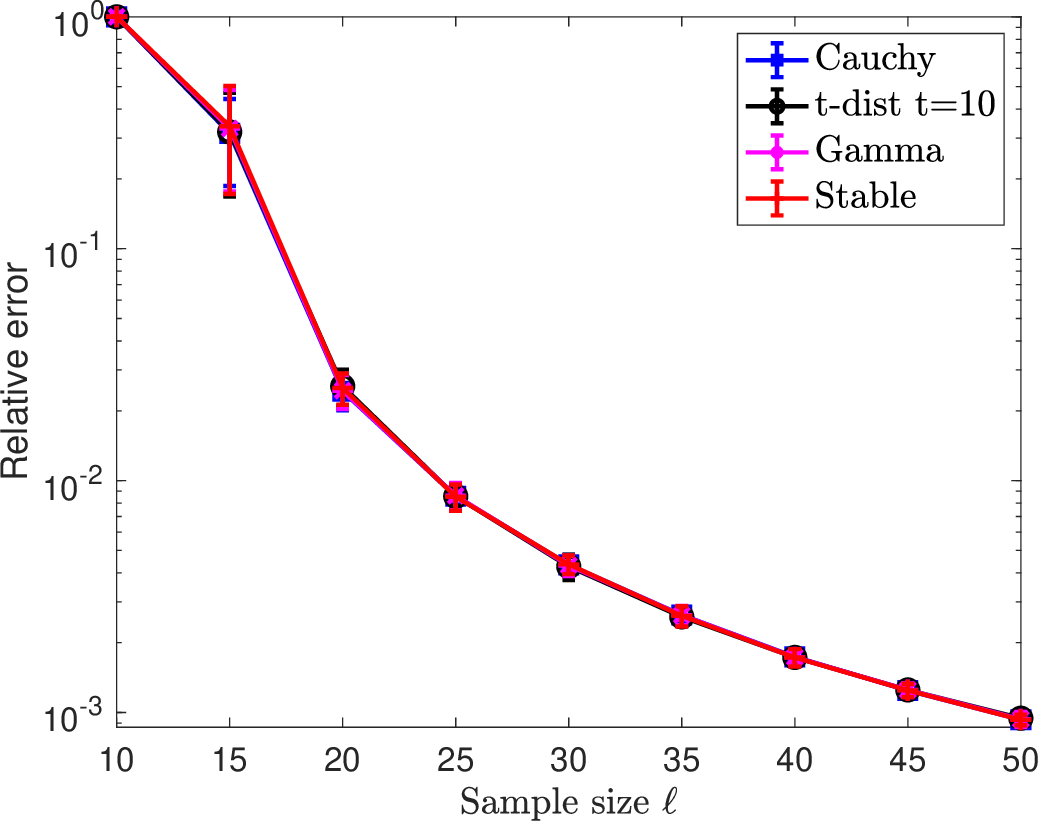}
    \includegraphics[scale=0.33]{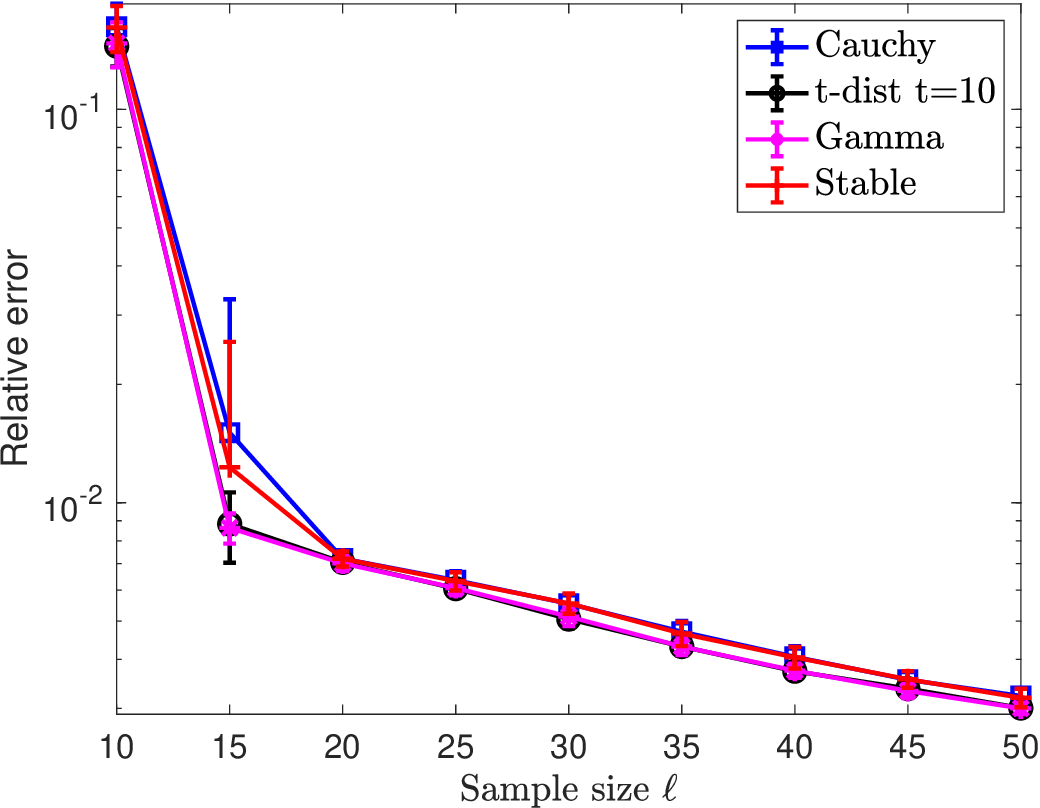}
    \caption{Relative error vs. sample size $\ell$ for the \texttt{FastDecay} (left) and \texttt{ControlledGap} (right). We use random matrices constructed using heavy-tailed distributions.}
    \label{fig:heavy}
\end{figure}

The first two distributions do not even satisfy \cref{def:randommat4} since they do not have finite moments. In \cref{fig:heavytail}, we plot the relative error vs. sample size $\ell$ for the two matrices \texttt{FastDecay} and \texttt{ControlledGap}. We see that the results are comparable across these distributions and to all other distributions previously considered. There is a slightly pronounced standard deviation of the error in the case of \texttt{ControlledGap} around index $15$. This suggests that the requirements of \cref{def:randommat4} are sufficient but not necessary. It would be interesting future work to develop error bounds for random matrices with heavy-tailed distributions.

\section{Nystr\"om Method}
\label{sm:Nystrom}

We consider a randomized Nyström method~\cite{williams2000using,drineas2005nystrom,halko2011finding,gittens2016revisiting} for the low-rank approximation of symmetric positive semidefinite (PSD) matrix $\mat A \in \R^{n \times n}$, i.e., $\mat A \succeq \mat{0}$. In this method, we draw a random matrix $\mat\Omega \in \R^{n\times \ell}$ and compute the sketch $\mat{Y} = \mat{A\Omega}$. Using this sketch, we can compute the Nystr\"om approximation
\[ \mat{A} \approx \mathat{A} := \mat{Y}(\mat\Omega\t\mat{Y})^\dagger \mat{Y}\t.\]
This low-rank approximation can be truncated using the \textit{truncated rank Nyström method} of~\cite{tropp2017fixed}, which uses $\mat{A} \approx [[\mathat{A}]]_k$, a best rank-$k$ approximation of the 
full Nyström approximation. A naive implementation of the Nystr\"om method can be numerically unstable, so we use a more stable version in~\cite[Algorithm 3]{tropp2017fixed}.

We give an example result of the guarantees using the Nystr\"om approximation. Let us partition the eigenvalues of $\mat{A}$ conformally as $$\mat{A} = \bmat{\mat{U}_k & \mat{U}_\perp} \bmat{\mat\Lambda_k \\  & \mat\Lambda_\perp}\bmat{\mat{U}_k\t \\ \mat{U}_\perp\t} ,$$
where the eigenvalues of $\mat{A}$ are arranged in decreasing order and the matrices $\mat\Lambda_k = \diag(\lambda_1,\dots,\lambda_k) \in \R^{k\times k}$ and $\mat\Lambda_\perp = \diag(\lambda_{k+1},\dots,\lambda_n) \in \R^{(n-k)\times (n-k)}$. We derive a result for the random matrix that satisfies~\cref{def:randommat2}; results for other models can be derived similarly.

\begin{theorem}[Nystr\"om, Independent sub-Gaussian columns] Let $\mat{\Omega} \in \R^{n \times \ell}$ be a random matrix satisfying \cref{def:randommat2}. Given user-specified parameters $ 0 < \delta < 1$ and $\varepsilon > 0$,  let the number of samples satisfy~\eqref{eqn:ellbound2}. If $\mathat{A} = \mat{Y}(\mat\Omega\t\mat{Y})^\dagger \mat{Y}\t$ is the Nystr\"om approximation with input $\mat\Omega$, then with a probability of failure at most $\delta$, 
\[\|\mat{A} - \mathat{A} \|^2 \leq \|\mat\Lambda_\perp\|_2  + \frac{1}{(1-\varepsilon)^2\ell}\left( \sqrt{\ell}\|\mat\Lambda_\perp\|_2^{1/2} + C_{\rm CB}K_{\rm C}^2(\sqrt{\trace(\mat\Lambda_\perp) } +  V_\delta\|\mat\Lambda_\perp\|_2^{1/2})  \right)^2. \]
Here, $C_{\rm CS}$ and $C_{\rm CB}$ are absolute constants.
\end{theorem}
\begin{proof}

Let $\Oh_2 = \mat{U}_\perp\t\mat\Omega$ and $\Oh_1 = \mat{U}_k\t \mat\Omega$. By \cref{thm:subgauss2}, $\Oh_1$ has full row rank with probability at least $1-\delta$. Condition on this event,  
 and apply~\cite{gittens2016revisiting}, with $q=1$ to get 
\begin{equation}\label{eqn:nystrominter} \|\mat{A} - \mathat{A}\|_2 \leq \|\mat\Lambda_\perp\|_2 + \|\mat\Lambda_\perp^{1/2}\Oh_2{\Oh_1^\dagger}\|_2^2.  \end{equation} 
Therefore, by \cref{thm:subgauss2}, with probability at least $1-\delta$

\[ \|\mat\Lambda_\perp^{1/2}\Oh_2{\Oh_1^\dagger}\|_2 \leq  \frac{ \|\mat\Lambda_\perp\|_2^{1/2}}{(1-\varepsilon)\sqrt{\ell}}\left( \sqrt{\ell} + C_{\rm CB}K_{\rm C}^2(\sqrt{\sr(\mat\Lambda_\perp^{1/2}) } +  {V_\delta})  \right). \]
Use $\sr(\mat\Lambda_\perp^{1/2}) = \trace(\mat\Lambda_\perp)/\|\mat\Lambda_\perp\|_2$ and plug into~\eqref{eqn:nystrominter}.
\end{proof}

In this section, we consider the performance of the Nystr\"om approximation using different random matrices.

\begin{figure}[!ht]
    \centering
    \includegraphics[scale=0.25]{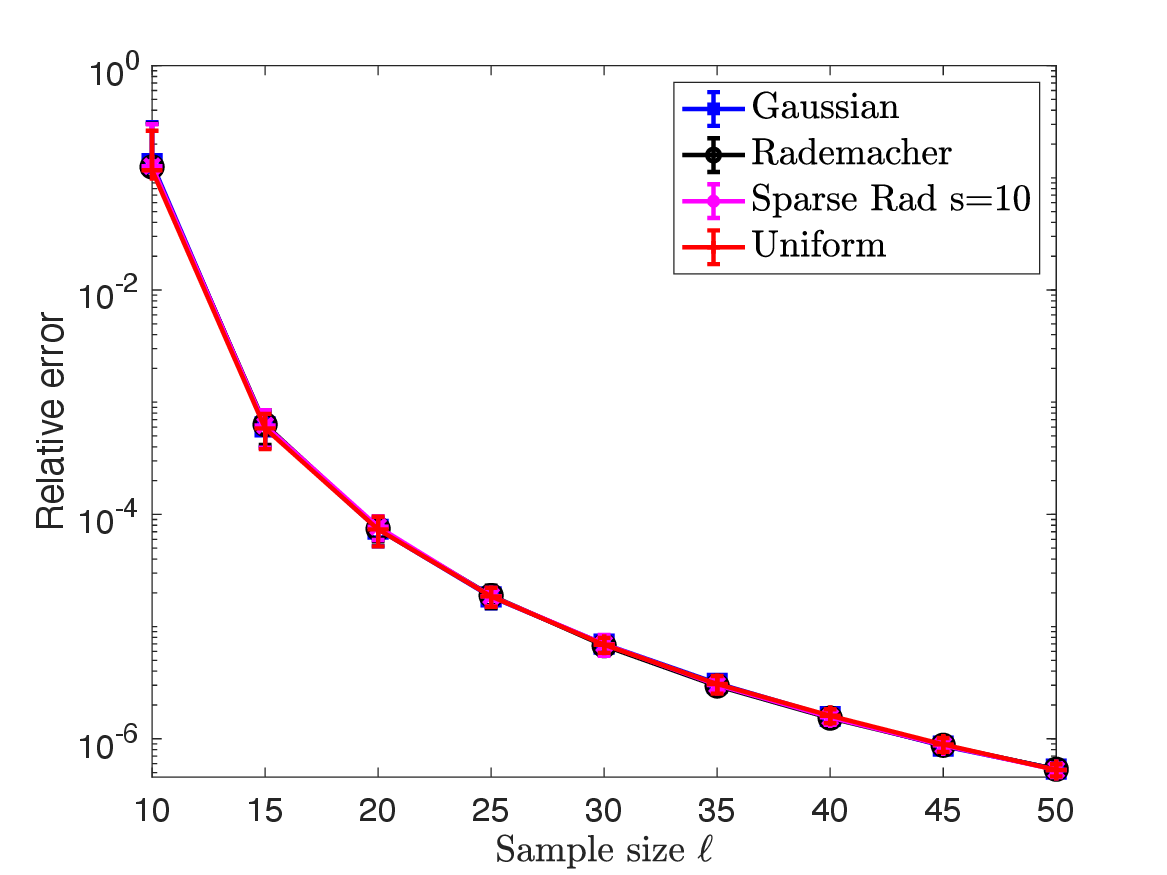}
    \includegraphics[scale=0.25]{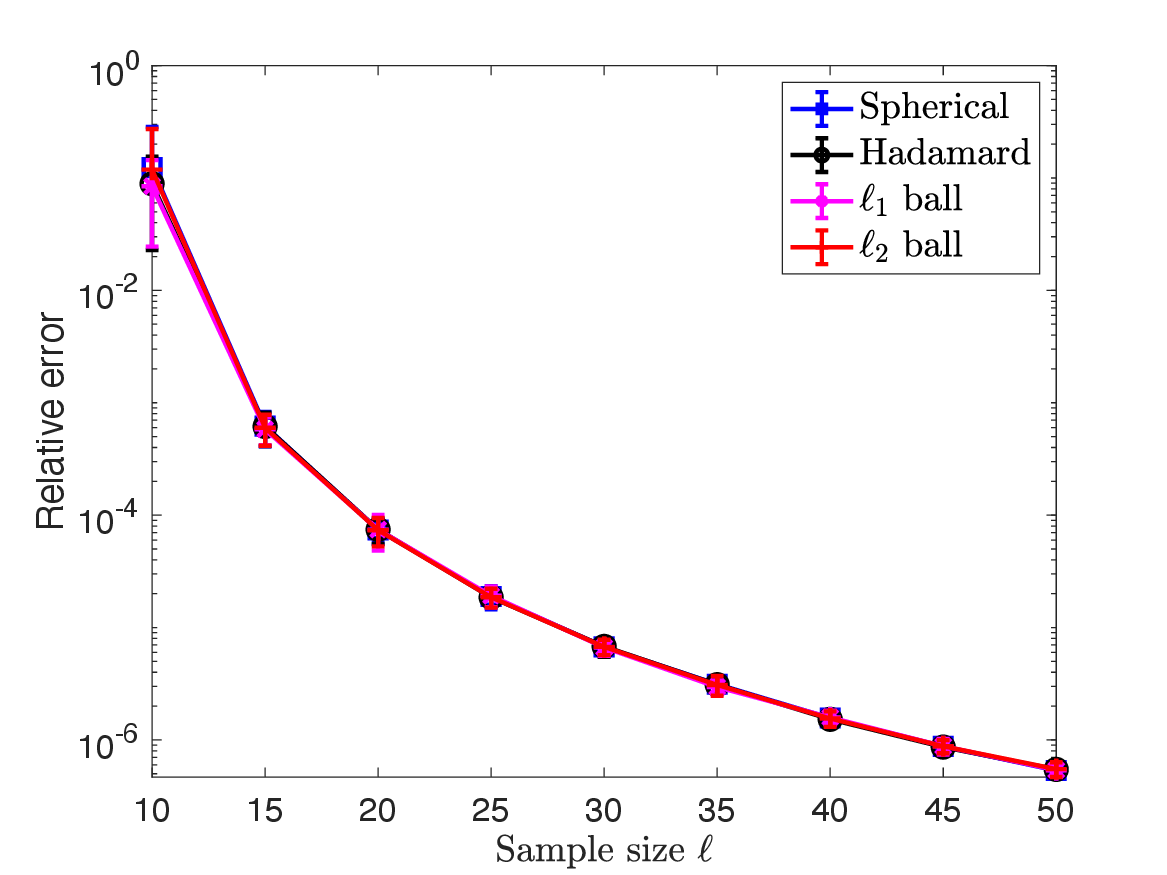}
   \includegraphics[scale=0.25]{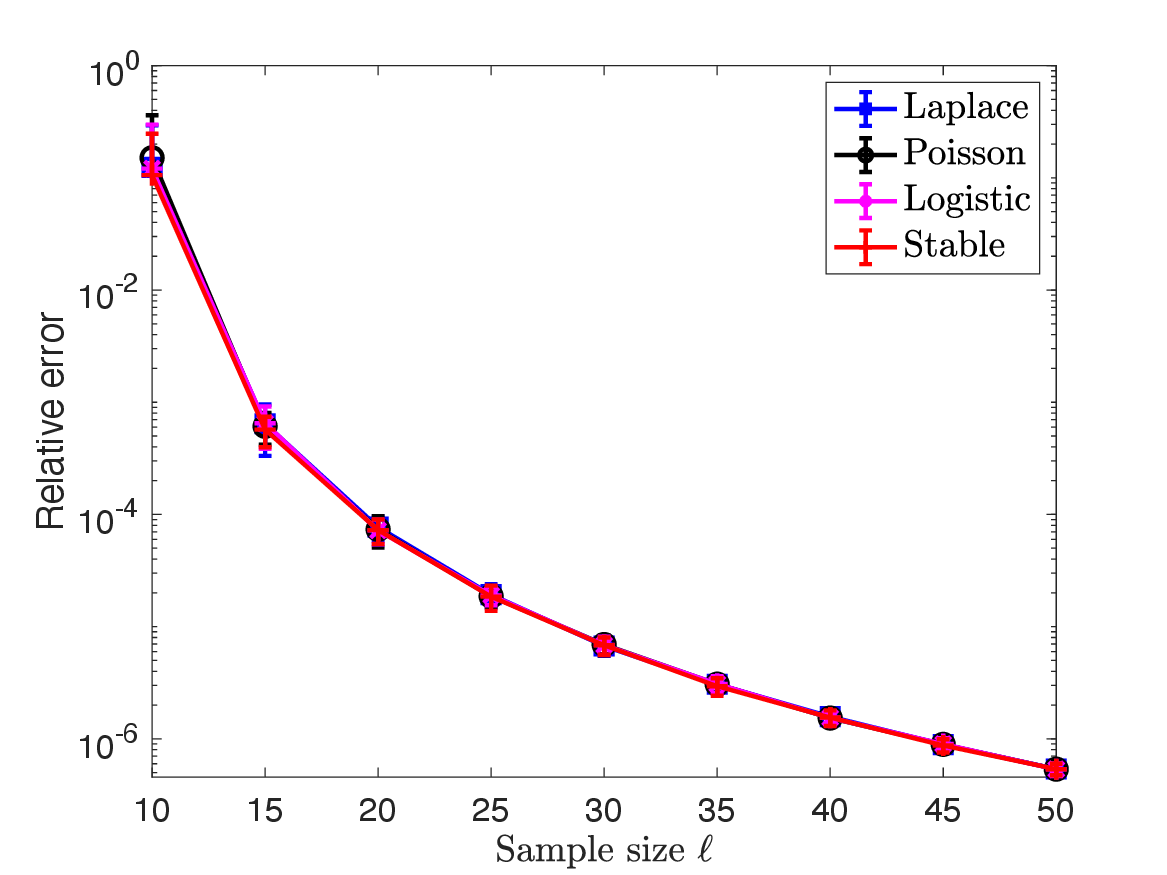}
    \caption{Relative error (RE) vs. sample size $\ell$ for 
    truncated rank Nyström low-rank approximation of the
    \texttt{FastDecayPSD} test matrix for each set of random matrices.}
    \label{fig:fastdecay}
\end{figure}
\paragraph{Test matrices} Here, we will consider the following two PSD test matrices $\mat A\in \R^{n \times n}$.
\begin{enumerate}
 \item \texttt{FastDecayPSD}: The matrix $\mat{A} \in \R^{n\times n}$ is constructed through its eigendecomposition $\mat{A} = \mat{U\Sigma^2 U}\t.$ The matrix $\mat{U}$ is an orthogonal matrix, first generated randomly and then computing the QR factorization. The eigenvalue matrix $\mat\Sigma^2$ is constructed using
    \[ \mat\Sigma = \diag(\underbrace{1,\dots,1}_{r}, 2^{-d},3^{-d},\dots, (n-r+1)^{-d}), \]
    where $n = 256$, $d=2$ is the degree of decay and $r=10$.
    \item \texttt{Abalone}: Abalone is a real-world dataset from the University of California, Irvine (UCI) Machine Learning Repository~\cite{KelLN}. It contains information about $8$ attributes (features, characteristics, physical measurements) for each of the $4177$ abalones, which can be used to predict the age of the abalone. The entries of the matrix $\mat{A}$ are generated based on the graph Laplacian. First we generate the matrix $\mat{W}$ as $[\mat{W}]_{i,j}=\kappa(\vec{x}_i,\vec{x}_j)$ for $1 \leq i,j\leq n$, where $\kappa(\vec{x},\vec{y})$ is the kernel $\exp(-\|\vec{x}-\vec{y}\|_2^2)$. The vectors $\{\vec{x}_i\}_{i=1}^{4177}$ correspond to the feature vectors. The graph Laplacian is then defined as $\mat{L} = \mat{I} - \mat{D}^{-1/2}\mat{WD}^{-1/2}$, where $\mat{D} = \mat{We}$ and $\vec{e}$ is a vector of ones. We take the matrix $\mat{A} = \mat{D}^{-1/2}\mat{WD}^{-1/2}$.  
    
\end{enumerate}
In \cref{fig:fastdecay}, we plot the results for the \texttt{FastDecayPSD} matrix; similarly in \cref{fig:nyst_abalone}, we plot the results for the \texttt{Abalone} matrix. The three panels in each plot correspond to the three sets of distributions discussed in Section~\ref{ssec:setup}; as before, we plot the mean over $100$ realizations and the error bar denotes one standard deviation.  For the \texttt{FastDecayPSD} matrix, the number of samples is between $10$ and $50$ and the error decays sharply. All the distributions we explored have similar performance. For the \texttt{Abalone} matrix, the sample size $\ell$ varies from $100-200.$ Note that since the number of rows ($4177$) in the \texttt{Abalone} matrix is not a power of $2$, we skip the results for the Hadamard distribution. We see that the Nystr\"om algorithm has a similar performance for all the distributions.
\begin{figure}[!ht]
    \centering
    \includegraphics[scale=0.255]{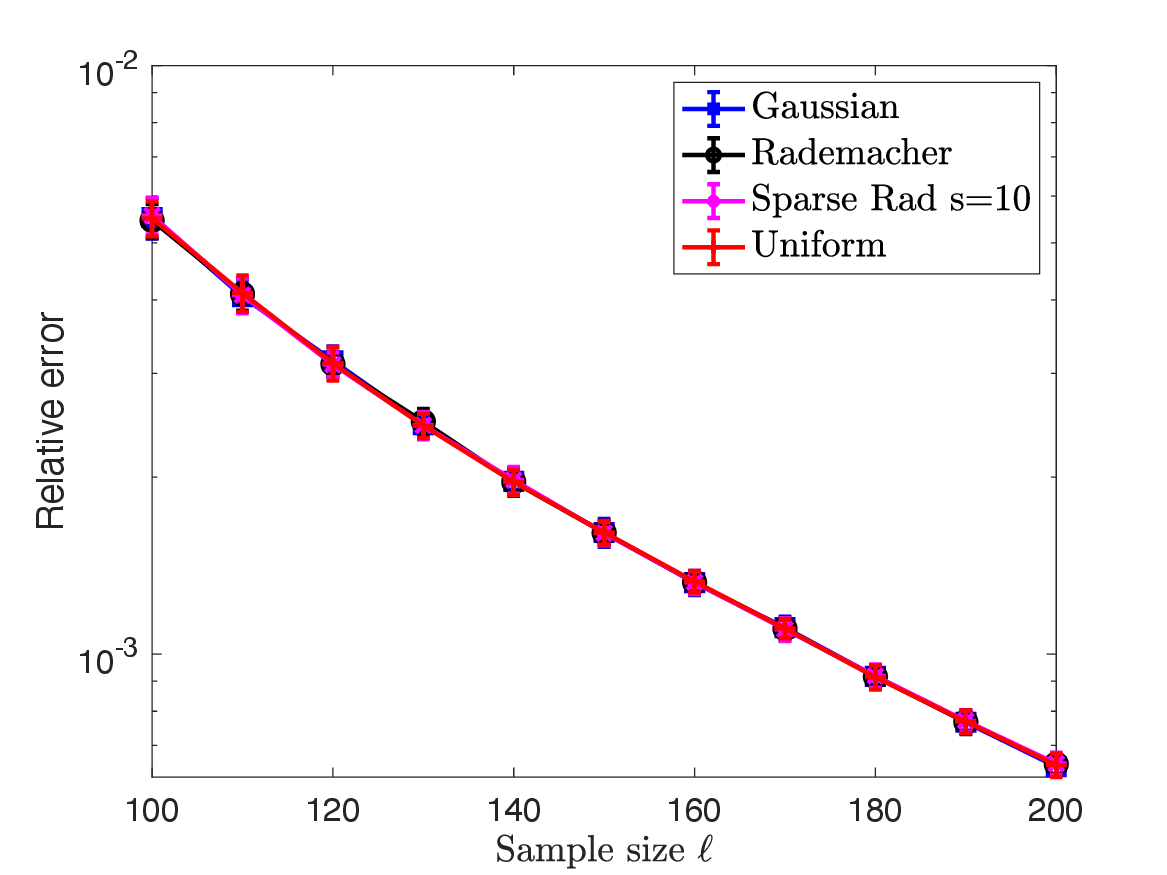}
    \includegraphics[scale=0.255]{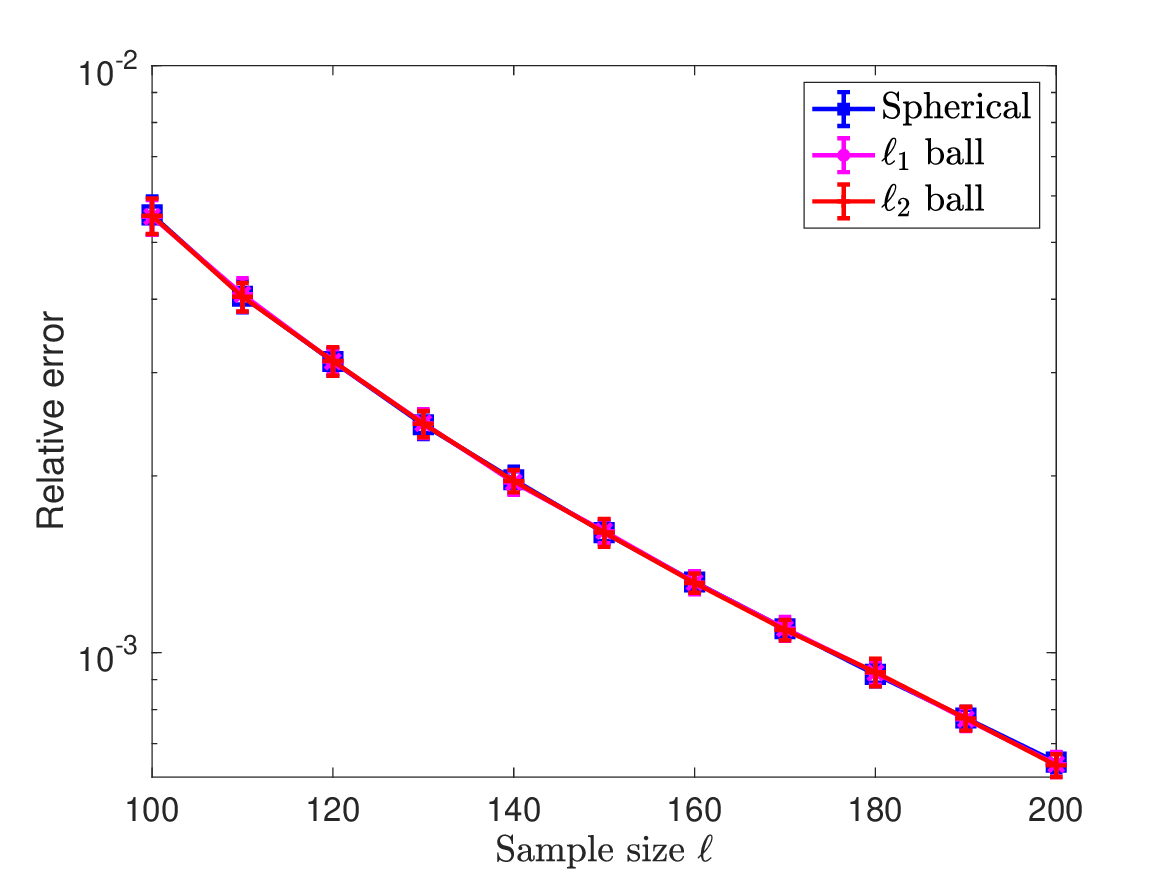}
    \includegraphics[scale=0.255]{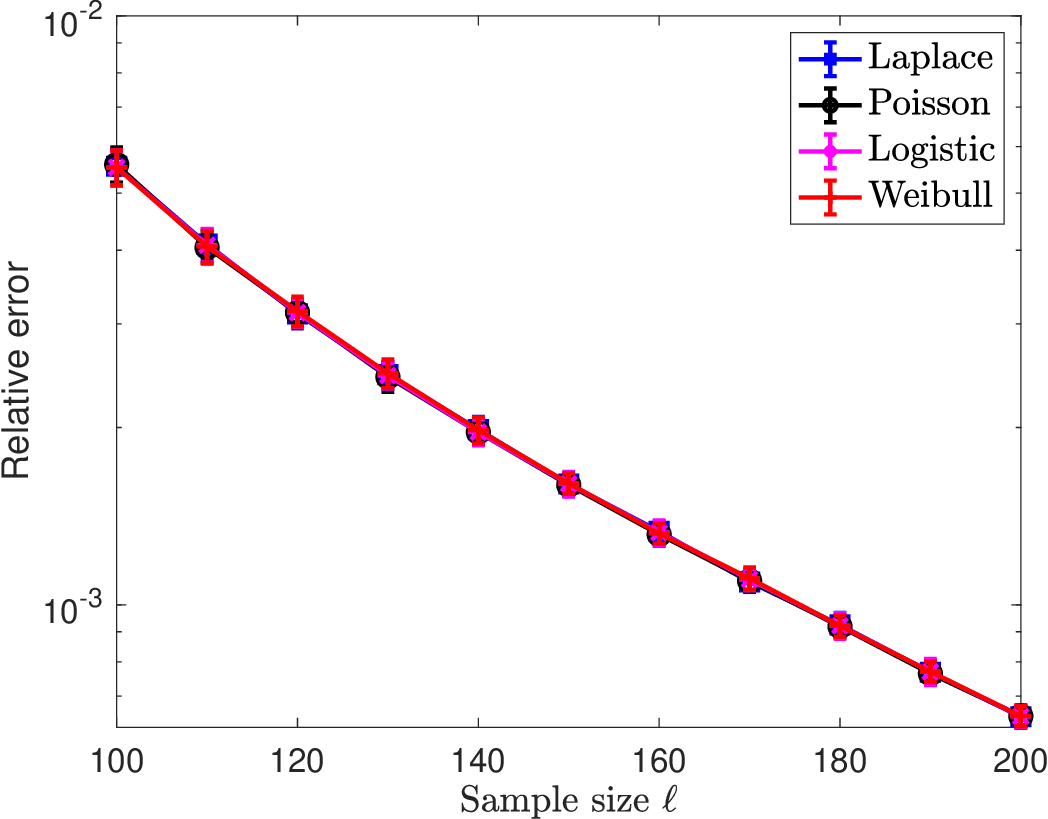}
    \caption{Relative error (RE) vs. sample size $\ell$ for truncated rank Nyström low-rank approximation of the \texttt{Abalone} test matrix for each set of random matrices except the Hadamard distribution.}
    \label{fig:nyst_abalone}
\end{figure}

\bibliography{refs}
\bibliographystyle{abbrv}
\end{document}

%% file: subgauss_shared.tex
\usepackage{lipsum}
\usepackage{amsfonts}
\usepackage[T1]{fontenc}

\usepackage{graphicx}
\usepackage{epstopdf}
\usepackage{algorithmic}
\ifpdf
  \DeclareGraphicsExtensions{.eps,.pdf,.png,.jpg}
\else
  \DeclareGraphicsExtensions{.eps}
\fi

\makeatletter
\newcommand*{\addFileDependency}[1]{
  \typeout{(#1)}
  \@addtofilelist{#1}
  \IfFileExists{#1}{}{\typeout{No file #1.}}
}
\makeatother

\usepackage{enumitem}
\setlist[enumerate]{leftmargin=.5in}
\setlist[itemize]{leftmargin=.5in}

\newsiamremark{remark}{Remark}
\newsiamremark{hypothesis}{Hypothesis}
\crefname{hypothesis}{Hypothesis}{Hypotheses}
\newsiamthm{claim}{Claim}

\headers{Randomized low-rank approximations}{A. K. Saibaba and A. Mi\k{e}dlar}

\title{Randomized low-rank approximations beyond Gaussian random matrices\thanks{Submitted to the editors DATE.
\funding{The work of A.K.S. was funded in part by the National Science Foundation through the awards DMS-1845406 and DMS-1745654. The work of A. Międlar was supported by the National Science Foundation through the awards DMS-2144181 and DMS-2324958.}}}

\author{ Arvind K.\ Saibaba\thanks{Department of Mathematics, North Carolina State University  
  (\email{asaibab@ncsu.edu}, \url{https://asaibab.math.ncsu.edu}).} 
\and Agnieszka Mi\k{e}dlar\thanks{Department of Mathematics, Virginia Tech, Blacksburg, VA 
  (\email{amiedlar@vt.edu}, \url{https://personal.math.vt.edu/amiedlar/}).}}

\usepackage{amsopn}
\DeclareMathOperator{\diag}{diag}

\newcommand{\R}{\mathbb{R}}

\newcommand{\sr}{\mathsf{sr}}
\newcommand{\intdim}{\mathsf{intdim}}
\renewcommand{\vec}[1]{\mathbf{#1}}
\newcommand{\mat}[1]{\mathbf{#1}}
\newcommand{\mathat}[1]{\widehat{\mathbf{#1}}}

\newcommand{\B}[1]{\boldsymbol{#1}}

\newcommand{\rad}{\mathsf{rad}}
\newcommand{\inner}[2]{\langle #1, #2\rangle}
\renewcommand{\t}{^\top}

\newcommand{\mc}[1]{\mathcal{#1}}
\newcommand{\bmat}[1]{\begin{bmatrix}#1\end{bmatrix}}

\newcommand{\rank}{\mathsf{rank}\,}
\newcommand{\trace}{\mathsf{trace}\,}
\newcommand{\uninorm}[1]{{\left\vert\kern-0.25ex\left\vert\kern-0.25ex\left\vert #1    \right\vert\kern-0.25ex\right\vert\kern-0.25ex\right\vert}}

\newcommand{\Oh}{\widehat{\B\Omega}}

\newcommand{\prob}[1]{\mathbb{P}\left\{ #1 \right\}}
\newcommand{\expect}[1]{\mathbb{E}\left[ #1 \right]}

\floatname{algorithm}{Procedure}
